\documentclass[a4paper,11pt]{amsart} 
\usepackage{amscd,amssymb}
\usepackage{amsfonts}
\usepackage{graphicx,epsf}
\usepackage{enumerate,color}
\usepackage[mathcal]{euscript}
\usepackage[dvipsnames]{xcolor}
\usepackage{hyperref}

\numberwithin{equation}{section}
\usepackage{epstopdf}
\usepackage{epsfig}
\usepackage{algorithm}
\usepackage{algpseudocode}
\newcommand*\Let[2]{\State #1 $\gets$ #2}

\newtheorem{theorem}{Theorem}[section]
\newtheorem{lemma}[theorem]{Lemma}
\newtheorem{proposition}[theorem]{Proposition}

\newtheorem {remark}[theorem]{Remark}

\newtheorem {problem}[theorem]{Inverse Problem}

\newcommand{\R}{\mathbb{R}}

\newcommand{\supp}{\mathrm{supp}}

\newcommand{\xibf}{\mbox{\boldmath $\xi$}}
\newcommand{\xbf}{\mbox{\boldmath $x$}}

\DeclareMathOperator{\bndry}{ {\partial\Omega} }

\newcommand{\tredbf}[1]{\textcolor{red}{\textbf{#1}}}

\begin{document}
\title[Revealing cracks with surface data]{Revealing cracks inside conductive bodies by electric surface measurements}
\author[A. Hauptmann, M. Ikehata, H. Itou, and S. Siltanen]{Andreas Hauptmann$^1$, Masaru Ikehata$^2$, Hiromichi Itou$^3$, and Samuli Siltanen$^4$}
\thanks{\noindent $^1$ Department of Computer Science, University College London, United Kingdom}
\thanks{$^2$ Laboratory of Mathematics, Graduate School of Engineering, Hiroshima University, Japan}
\thanks{$^3$ Department of Mathematics, Tokyo University of Science, Japan}
\thanks{$^4$ Department of Mathematics and Statistics, University of Helsinki, Finland}
\date{\today}

\begin{abstract}
An algorithm is introduced for using electrical surface measurements to detect and monitor cracks inside a two-dimensional conductive body. The technique is based on transforming the probing functions of the classical {\it enclosure method} by the Kelvin transform. The transform makes it possible to use virtual discs for probing the interior of the body using electric measurements performed on a flat surface. Theoretical results are presented to enable probing of the full domain to create a profile indicating cracks in the domain. Feasibility of the method is demonstrated with a simulated model of attaching metal sheets together by resistance spot welding.	
\end{abstract}
\maketitle

\section{Introduction}
\noindent
Resistance Spot Welding (RSW) is an established technique for joining two metallic pieces together in industrial assembly lines. The principle is  simple: an electric current is applied while compressing the metallic materials. The metal melts by the resistance heat, and a molten nugget is produced near the faying surface. This results in the pieces being joined together. 

RSW is often performed autonomously by robots. Therefore, an automatic quality control would be useful. We extend the approach introduced in \cite{Ikehata2016}, a theoretical study of using electrical boundary measurements for probing the interior of a conductive body for cracks. The technique is based on transforming the indicator function of the classical {\it enclosure method} \cite{Ikehata2000c} by the Kelvin transform. The transform makes it possible to use virtual discs for probing the interior of the body using electric measurements performed on a flat surface. For a fixed crack location, the method only needs one boundary measurement. 

There are two main novelties in this paper compared to \cite{Ikehata2016}. We provide additional theoretical foundations for the method. Furthermore, we demonstrate computationally that the Kelvin transformed enclosure method can be used for robust detection of multiple crack locations. The computational procedure seeks to compute a profile of the metal slab, where crack tips are indicated as local maxima.

Mathematically the problem under consideration can be modeled as follows. First of all, for simplicity we focus this study on a two-dimensional formulation. Then the two metal slabs, stacked on top of each other, are modeled by a bounded and rectangular domain $\Omega\subset\R^{2}$.
We fix the Cartesian coordinates such that $\Omega=]0,\,a[\times\,]0,b[$ with $a>0$ and $b>0$. 
Further, let $c\in]0,b[$ denote the vertical location of the boundary between the two metal pieces, as illustrated in Figure \ref{fig:Intro}. The part where the plates are not joined is denoted by $\Sigma\subset ]0,a[\times \{c\}$, which we call the crack in $\Omega$, additionally we denote the upper/lower parts by $\Omega^{\pm}=\Omega\cap\{x\in\R^2\,\vert\,\pm(x_2-c)>0\}.$ Before the RSW procedure $\Sigma$ will strictly divide the two plates, after the procedure is started the plates will join and $\Sigma$ will consist of disjoint sets of cracks. As illustrated in Figure \ref{fig:Intro}, the profile computed by the Kelvin transformed enclosure method then indicates the crack tips after RSW as local maxima.

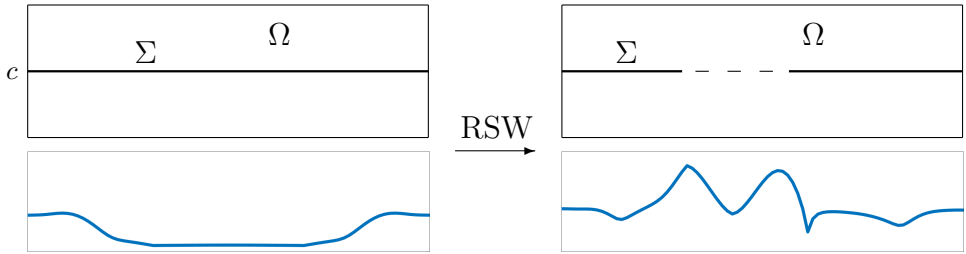
\begin{figure}[h!]
\centering
\begin{picture}(350,100)

\put(0,50){\Large{\line(1,0){150}}}
\put(0,50){\Large{\line(0,1){50}}}
\put(150,50){\Large{\line(0,1){50}}}
\put(0,100){\Large{\line(1,0){150}}}
{\linethickness{0.3mm}
\put(0,75){\Large{\line(1,0){150}}}

}

\put(40,78){\large{$\Sigma$}}
\put(90,85){\large{$\Omega$}}
\put(-8,72){\small{$c$}}

\put(162,50){\large{RSW}}
\put(160,45){\Large{\vector(1,0){30}}}

\put(200,50){\Large{\line(1,0){150}}}
\put(200,50){\Large{\line(0,1){50}}}
\put(350,50){\Large{\line(0,1){50}}}
\put(200,100){\Large{\line(1,0){150}}}
{\linethickness{0.3mm}
\put(200,75){\Large{\line(1,0){45}}}
\put(285,75){\Large{\line(1,0){65}}}
}
\multiput(250,75)(12,0){3}{\line(1,0){5}}

\put(220,78){\large{$\Sigma$}}
\put(290,85){\large{$\Omega$}}

\put(-3,5){\includegraphics[width=158pt]{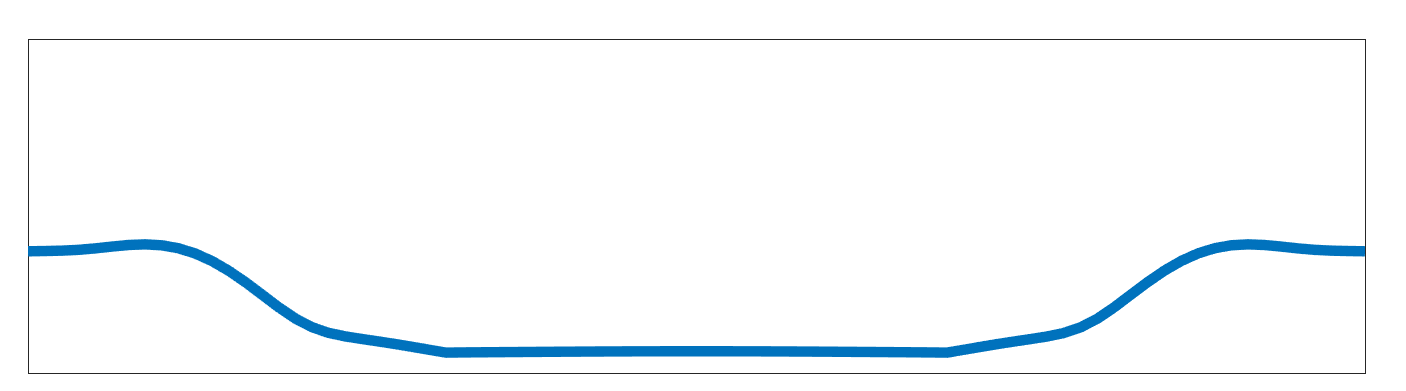}}
\put(197,5){\includegraphics[width=158pt]{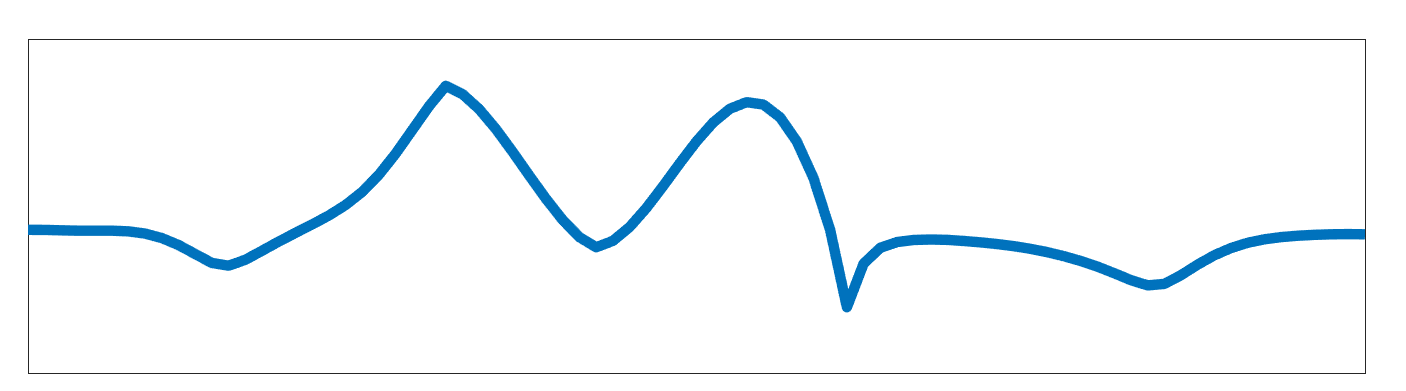}}

\end{picture}
\caption{\label{fig:Intro} Illustration of the problem setting: $\Omega$ models the two stacked metal slabs to be joined by resistance spot welding (RSW) and $\Sigma$ denotes the set of cracks between the two plates. During the welding procedure the plates join and the task is to detect the edges of the remaining cracks to monitor the quality of welding. The second row shows the obtained probing profile indicating the crack tips with the proposed method.}
\end{figure} 

The physical processes can be modeled by a conductivity equation. Let us assume that the two plates consist of a common isotropic conductive medium having a constant conductivity $\sigma$.
Without loss of generality, one may assume that $\sigma=1$.
Let $ \mbox{\boldmath $\nu$} $ denote the unit outward normal vector field on $\partial\Omega $.
Given a mean-free electric current density $g\in L^2(\partial\Omega)$, i.e. $\displaystyle\int_{\partial\Omega}g\; {\rm d}s_{\mbox{\boldmath $x$}}=0$, then the resulting voltage inside $\Omega$ satisfies

\begin{eqnarray}
\left\{ \begin{array}{ll}
\vspace{0.3cm}
\Delta u=0 & \text{in $\Omega\setminus\Sigma$,}\\
\vspace{0.3cm}
\frac{\partial u}{\partial x_2}=0 & \text{on $\Sigma$},\\
\frac{\partial u}{\partial\mbox{\boldmath $\nu$}}=g & \text{on $\partial\Omega$.}
\end{array}
\right.\label{eq:0}
\end{eqnarray}
Here we assume that $u\in H^1(\Omega\setminus\Sigma)$ is the weak solution of \eqref{eq:0} in the sense that it satisfies
\begin{eqnarray}
\int_{\Omega\setminus\Sigma} \nabla u \cdot \nabla \varphi\;{\rm d}\mbox{\boldmath $x$}=\int_{\partial\Omega} g\varphi\;{\rm d}s_{\mbox{\boldmath $x$}}
\label{eq:1}
\end{eqnarray}
for all $ \varphi\in H^{1}(\Omega\setminus\Sigma) $. Let us further assume $\int_{\partial\Omega}u\; {\rm d}s_{\mbox{\boldmath $x$}}=0$ to assure uniqueness.
The condition on $\Sigma$ in \eqref{eq:0} implies that there is no current flux through the cracks. Thus, it is the place where the two plates are not joined.

Succeeding to \cite{Ikehata2016}, in this paper we consider the following inverse problem.
\begin{problem}
Fix $ g\neq 0 $.
Extract information about the location of $ \Sigma $ from the knowledge of a single set of Cauchy data $(g,u)$ on $ \partial \Omega $.
\end{problem}

This is an example of a so-called inverse crack problem using a single set of the Cauchy data.
Note that $\Sigma$ plays the role of the set of unknown cracks.
We emphasize that the problem asks to seek an {\it extraction procedure} of information about the geometry of $\Sigma$.
Recently an extraction procedure has been established in \cite{Ikehata2016} by combining the enclosure method \cite{Ikehata2000c} with a Kelvin transform. The main result of \cite{Ikehata2016} is concerned with the asymptotic behavior of an {\it indicator function} which can be computed by using the Cauchy data $u$ and $g$ on $\partial\Omega$.

The inverse crack problem considered in this paper is of the type: 
we know that the crack is located on a known line.
Alternatively, if a crack lies on an unknown line (or plane in three-dimensions) and completely embedded in the domain, uniqueness and stability results using a single set of the Cauchy data have been established in \cite{Alessandrini1997}.
See also \cite{Alessandrini2013} for a recent study on the stability issue. 
A reconstruction formula of the crack on a single unknown plane embedded in a reference medium which is based on the reciprocity gap method with a single set of Cauchy data has been established in \cite{Andrieux1996}.
Note that therein they assume that the integral of the jump of the potential over the crack does not vanish. 
Furthermore, in \cite{Isakov1995} a uniqueness theorem has been established for the {\it surface breaking crack} which includes our case as a special case.
See also the classical result in \cite{Friedman1989a} and a reconstruction method \cite{Aparicio1996}.  
These are concerned with inverse crack problems with a single or two sets of Cauchy data.
Another approach employs the Neumann-to-Dirichlet map (or Dirichlet-to-Neumann map) or its localization on the surface of the domain. In this case the fundamental uniqueness result for the crack having general shape embedded in the domain is established in \cite{Eller1996}. 
Reconstruction methods for cracks having a general shape are the probe method \cite{Ikehata2006, Ikehata2003c}, the enclosure method \cite{Ikehata2008} and the factorization method \cite{Boukari2013, Bruhl2001b}.

The first computational implementation of the enclosure method using the Dirichlet-to-Neumann map or Neumann-to-Dirichlet map was done by \cite{Ikehata2000a} and \cite{Bruhl2000} for EIT to reconstruct the convex hull of an inclusion. For an implementation of the enclosure method with non convex inclusions see \cite{Ikehata2004}.
The enclosure method with a single set of the Dirichlet and Neumann data is demonstrated in \cite{Ikehata2002} and \cite{Ikehata2002c} which treated inclusions and cavities, respectively.

This paper concentrates on a two-dimensional formulation of the crack detection problem. Therefore, our results are not directly applicable to the actual three-dimensional situation arising in practical spot welding. However, understanding the two-dimensional case both theoretically and computationally paves the way to a three-dimensional extension of the method.

The outline of this paper is as follows.  The theoretical background and and outline of the probing algorithm used for crack detection is given in Section \ref{sec:theoIntro}, where we also state the main result of this paper in Theorem \ref{th1} that lays the theoretical foundation to perform the probing algorithm between the cracks. In Section \ref{sec:proof} we give a proof of our main result. First we describe an expression of the indicator function and the behavior of the solution of \eqref{eq:0} around a tip of the crack.
Using those, we reduce the problem to studying the leading profile of an oscillatory integral with large parameter $\tau$ which is stated as Lemma \ref{lem1-3}.
The proof of Lemma \ref{lem1-3} is given in Subsection \ref{sec:proofLim}, which is an application of the method of {\it  steepest descent} (cf. \cite{Olver1997}).
In Section \ref{sec:numericalConsider} we discuss the implementation of the probing algorithm based on the Kelvin transformed enclosure method and propose a monitoring procedure during an idealized setting of RSW. We then present our computational results in Section \ref{sec:compExperiments} of a possible probing procedure during the welding process. Final conclusions are then presented in Section \ref{sec:conclusions}.

\section{A probing algorithm for crack detection}\label{sec:theoIntro}
Let us start by extending the definition of the crack geometry. We remind that the welding area is given by the line $[0,a]\times\{c\}$, we divide this area into the parts that are already joined and those are not, the cracks.
Then the section of already joined plates is given by
\[
W=[0,\,a]\times\{c\}\setminus\Sigma,
\]
where $\Sigma$ denotes the cracks and is a subset of $[0,\,a]\times\{c\}$, given by
\[
\Sigma=\bigcup_{j=0}^m[c_{2j}, c_{2j+1}]\times\{c\}
\]
with an integer $m\ge 1$ denoting the number of cracks, such that $0=c_0<c_1<\cdots<c_{2m}<c_{2m+1}=a$.
See Figure \ref{fig00} for an illustration of the general geometry under consideration.

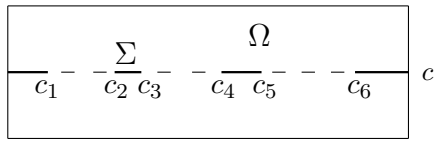
\begin{figure}[h!]
\centering
\begin{picture}(150,60)

\put(0,0){\Large{\line(1,0){150}}}
\put(0,0){\Large{\line(0,1){50}}}
\put(150,0){\Large{\line(0,1){50}}}
\put(0,50){\Large{\line(1,0){150}}}
{\linethickness{0.3mm}
\put(0,25){\Large{\line(1,0){15}}}
\put(40,25){\Large{\line(1,0){10}}}
\put(80,25){\Large{\line(1,0){15}}}
\put(130,25){\Large{\line(1,0){20}}}
}
\multiput(8,25)(12,0){3}{\line(1,0){5}}
\multiput(43,25)(13,0){3}{\line(1,0){5}}
\multiput(88,25)(11,0){4}{\line(1,0){5}}
\put(40,28){\large{$\Sigma$}}
\put(90,35){\large{$\Omega$}}
\put(155,22){\small{$c$}}
\put(10,17){$c_{1}$}
\put(36,17){$c_{2}\;c_{3}$}
\put(76,17){$c_{4}\;\;c_{5}$}
\put(127,17){$c_{6}$}
\end{picture}
\caption{\label{fig00} Illustration of the general setting under consideration. $\Omega$ denotes the metal slab. The crack set $\Sigma$ is divided into 4 separate cracks on the line $[0,a]\times\{c\}$ parametrized by $c_j$, $j=1,\dots,6$, that we want to detect.}
\end{figure} 

\subsection{Theoretical considerations}
In this paper, we make use of the following notation: $B_R(\mbox{\boldmath $x$})$ denotes the open disc
centered at $\mbox{\boldmath $x$}\in\mathbb{R}^2$ with radius $R$; $\mbox{\boldmath $e$}_1=(1,0)^T$, $\mbox{\boldmath $e$}_2=(0,1)^T$.
We recall the enclosure method combined with the Kelvin transform as established in \cite{Ikehata2016}. 
First we need a family of special solutions of the Laplace equation for a large parameter $\tau>0$.
Given $\mbox{\boldmath $\xi$}\in\R^2\setminus\overline{\Omega}$  define
\[
v_{\tau}(\mbox{\boldmath $x$};\mbox{\boldmath $\xi$})=\exp\left\{-\tau\frac{\mbox{\boldmath $x$}-\mbox{\boldmath $\xi$}}{\vert \mbox{\boldmath $x$}-\mbox{\boldmath $\xi$}\vert^2}\cdot(\mbox{\boldmath $e$}_2+i\mbox{\boldmath $e$}_1)\right\},\quad \mbox{\boldmath $x$}\in\R^2\setminus\{\mbox{\boldmath $\xi$}\}.
\]
This is a solution of the Laplace equation in $\R^2\setminus\{\mbox{\boldmath $\xi$}\}$ with the asymptotic behavior of $e^{-\tau/(2s)}v_{\tau}(\mbox{\boldmath $x$};\mbox{\boldmath $\xi$})$ for $s>0$ and $\tau\rightarrow\infty$:
\begin{eqnarray}
\lim_{\tau\rightarrow\infty}e^{-\tau/(2s)}v_{\tau}(\mbox{\boldmath $x$};\mbox{\boldmath $\xi$})
=
\left\{
\begin{array}{ll}
0 & \text{if $\mbox{\boldmath $x$}\in\mathbb{R}^2\setminus\overline{B_s (\mbox{\boldmath $\xi$}-s\mbox{\boldmath $e$}_2)}$,}\\
\\
\displaystyle
\infty & \text{if $\mbox{\boldmath $x$}\in B_s(\mbox{\boldmath $\xi$}-s\mbox{\boldmath $e$}_2)$.}
\end{array}
\right.
\label{eq:3}
\end{eqnarray}
Note that, on the circle $\partial B_s(\mbox{\boldmath $\xi$}-s\mbox{\boldmath $e$}_2)$ the function $e^{-\tau/(2s)}v_{\tau}(\mbox{\boldmath $x$};\mbox{\boldmath $\xi$})$ is highly oscillating as $\tau\rightarrow\infty$.

We then define an indicator function by
\begin{eqnarray}
I(\tau ;\mbox{\boldmath $\xi$}):=\int_{\partial\Omega}
\left(g(\mbox{\boldmath $x$})v_{\tau}(\mbox{\boldmath $x$};\mbox{\boldmath $\xi$})-u(\mbox{\boldmath $x$})\frac{\partial v_{\tau}(\mbox{\boldmath $x$};\mbox{\boldmath $\xi$})}{\partial\mbox{\boldmath $\nu$}_{\mbox{\boldmath $x$}}}\right)\; {\rm d}s_{\mbox{\boldmath $x$}}.
\label{eq:2}
\end{eqnarray}
Here we restrict the moving range of $\mbox{\boldmath $\xi$}$ to the segment $\displaystyle\Gamma_{\epsilon}=[0, a]\times\,\{b+\epsilon\}$ with a fixed positive number $\epsilon$ and let 
\begin{eqnarray}
s=(b+\epsilon-c)/2,
\label{eq:4}
\end{eqnarray}
which is in the middle point of $\mbox{\boldmath $\xi$}\in\Gamma_{\epsilon}$ from line $x_2=c$.
In \cite{Ikehata2016} it is shown that, if $\supp(g) \subset ]0,a[\times \{b\}$, see also condition A1.) in Theorem \ref{th1}, then the indicator function $I(\tau ;\mbox{\boldmath $\xi$})$ multiplied by $e^{-\tau/(2s)}$ as $\tau\rightarrow\infty$ has the following asymptotic behavior:
\begin{itemize}
\item If the projection of $\mbox{\boldmath $\xi$}$ onto the line $x_2=c$ coincides with a tip of crack $\Sigma$ in $\Omega$, then there exists an integer $N\ge 1$ and a positive number $A$ such that
\begin{eqnarray}
\lim_{\tau\rightarrow\infty}\tau^{(2N-1)/2}e^{-\tau/(2s)}|I(\tau;\mbox{\boldmath $\xi$})|=A.
\label{eq:5}
\end{eqnarray}
\item If the projection of $\mbox{\boldmath $\xi$}$ onto the line $x_2=c$ does not coincide with any tip of crack $\Sigma$ in $\Omega$, then the function $e^{-\tau/(2s)}I(\tau ;\mbox{\boldmath $\xi$})$ is exponentially decaying as $\tau\rightarrow\infty$.
See Figure \ref{fig00-1} for an illustration of this case.
\end{itemize}

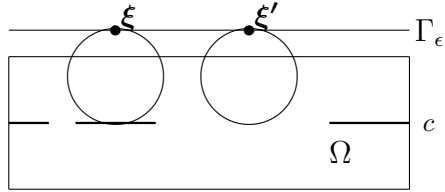
\begin{figure}[h!]
\centering
\begin{picture}(150,80)
\put(0,0){\Large{\line(1,0){150}}}
\put(0,60){\Large{\line(1,0){150}}}
\put(0,0){\Large{\line(0,1){50}}}
\put(150,0){\Large{\line(0,1){50}}}
\put(0,50){\Large{\line(1,0){150}}}
{\linethickness{0.3mm}
\put(0,25){\Large{\line(1,0){15}}}
\put(25,25){\Large{\line(1,0){30}}}
\put(120,25){\Large{\line(1,0){30}}}
}
\put(40,42.5){\circle{34.0}}
\put(37.5,57.5){\small{{$\bullet$}}}
\put(42,62){$\mbox{\boldmath $\xi$}$}

\put(90,42.5){\circle{34.0}}
\put(87.5,57.5){\small{{$\bullet$}}}
\put(92,62){$\mbox{\boldmath $\xi'$}$}

\put(120,10){\large{$\Omega$}}
\put(152,56){\large{$\Gamma_\epsilon$}}
\put(155,22){\small{$c$}}
\end{picture}
\caption{\label{fig00-1} Illustration of the case when the projection onto the line $x_2=c$ does not coincide with a tip of crack. Here $(\xi_1,c)\in\Sigma \backslash \{c_0,c_1,\dots,c_{2m+1}\}$ and $(\xi_1',c)\in W$.}
\end{figure}

Note that \eqref{eq:5} implies that the function $e^{-\tau/(2s)}I(\tau;\mbox{\boldmath $\xi$})$ is {\it truly algebraic} decaying as $\tau\rightarrow\infty$.
Thus by the {\it difference} of the decaying property of the function $e^{-\tau/(2s)}I(\tau ;\mbox{\boldmath $\xi$})$ as $\tau\rightarrow\infty$ one can identify all the tips of $\Sigma$ in $\Omega$.
This gives a qualitative identification procedure for the crack $\Sigma$.

The purpose of the present paper is to extend the method used in \cite{Ikehata2016} to a setting that is suitable for probing the whole domain and propose a numerical algorithm based on these results. For this purpose we take a modified approach as described above which is based on the decaying property of $e^{-\tau/(2s)}I(\tau ;\mbox{\boldmath $\xi$})$ with $s$ given by \eqref{eq:4}.

For this purpose we define the following function of $\mbox{\boldmath $\xi$}\in\Gamma_{\epsilon}$ given by
\[
s_{\Sigma}(\mbox{\boldmath $\xi$}):=\sup\,\{s>0\,\vert\,B_s(\mbox{\boldmath $\xi$}-s\mbox{\boldmath $e$}_2)\subset\R^2\setminus\Sigma\}.
\]
The value $s_{\Sigma}(\mbox{\boldmath $\xi$})$ at $\mbox{\boldmath $\xi$}\in\Gamma_{\epsilon}$ coincides with the radius
of the largest disc centered at $\mbox{\boldmath $\xi$}-s\mbox{\boldmath $e$}_2$ with radius $s$ whose exterior encloses $\Sigma$.
Note that $s$ given by \eqref{eq:4} coincides with $s_{\Sigma}(\mbox{\boldmath $\xi$})$ if the projection of $\mbox{\boldmath $\xi$}$ onto the line $x_2=c$ belongs to $\Sigma$, that is the case for $\xibf$ in Figure \ref{fig00-1}.  
Furthermore, equation \eqref{eq:5} implies, that if the projection of $\mbox{\boldmath $\xi$}$ onto the line $x_2=c$ coincides with a tip of crack $\Sigma$ in $\Omega$
we simply have
\begin{eqnarray}
\lim_{\tau\rightarrow\infty}\frac{\log\vert I(\tau ;\mbox{\boldmath $\xi$})\vert}{\tau}=\frac{1}{2s}.
\label{eq:6}
\end{eqnarray}
The theoretical purpose of this paper is to extend this formula to the case when the projection of $\xibf$ onto the line $x_2=c$ belongs to $W$, as shown in Figure \ref{fig00-2}. Let us now state our main result and include $s_{\Sigma}(\mbox{\boldmath $\xi$})$ in \eqref{eq:6} instead of $s$.

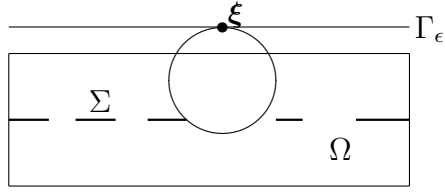
\begin{figure}[h!]
\centering
\begin{picture}(150,80)
\put(0,0){\Large{\line(1,0){150}}}
\put(0,60){\Large{\line(1,0){150}}}
\put(0,0){\Large{\line(0,1){50}}}
\put(150,0){\Large{\line(0,1){50}}}
\put(0,50){\Large{\line(1,0){150}}}
{\linethickness{0.3mm}
\put(0,25){\Large{\line(1,0){15}}}
\put(25,25){\Large{\line(1,0){15}}}
\put(52,25){\Large{\line(1,0){14.5}}}
\put(100,25){\Large{\line(1,0){10}}}
\put(130,25){\Large{\line(1,0){20}}}
}
\put(30,28){\large{$\Sigma$}}
\put(80,40){\circle{40}}
\put(77.5,57.5){\small{{$\bullet$}}}
\put(82,62){$\mbox{\boldmath $\xi$}$}

\put(120,10){\large{$\Omega$}}
\put(152,56){\large{$\Gamma_\epsilon$}}
\end{picture}
\caption{\label{fig00-2}  Illustration of the case when the projection onto the line $x_2=c$ belongs to $W$ and the intersection of the closure of the disc $B_s(\mbox{\boldmath $\xi$}-s\mbox{\boldmath $e$}_2)\vert_{s=s_{\Sigma}(\mbox{\boldmath $\xi$})}$ with $\Sigma$
consists of a single tip of $\Sigma$. }
\end{figure}

\begin{theorem}\label{th1}
Let the mean-free Neumann data $g\in L^2(\partial \Omega)$ be such that either is fulfilled:
\begin{itemize}
\item[A1.)] $\supp(g) \subset ]0,a[ \times \{b\}$.
\item[A2.)] $\supp(g) \subset  ( \partial\Omega \cap \{|x_2-c|>\delta \} ) \backslash  (\ B_\delta(O) \cup B_\delta((0,b)) \cup B_\delta((a,b)) \cup B_\delta((a,0))\ )$
for some $\delta > 0$ and 
$$\int_{\partial\Omega \cap \{x_2>c\} } g \; {\rm d}s_{\mbox{\boldmath $x$}} \neq 0 \text{ or } \int_{\partial\Omega \cap \{x_2<c\} } g \; {\rm d}s_{\mbox{\boldmath $x$}} \neq 0,$$ 
\end{itemize}
Let further $\mbox{\boldmath $\xi$}$ satisfy that the intersection of $\overline{B_{s_{\Sigma}(\mbox{\boldmath $\xi$})}(\mbox{\boldmath $\xi$}-s_{\Sigma}(\mbox{\boldmath $\xi$})\mbox{\boldmath $e$}_2)}$ with the crack $\Sigma$ in $\Omega$ consists of a single point, namely there exists a $j\in\{1,\cdots,2m\}$ such that
\begin{eqnarray}
\overline{B_{s_{\Sigma}(\mbox{\boldmath $\xi$})}(\mbox{\boldmath $\xi$}-s_{\Sigma}(\mbox{\boldmath $\xi$})\mbox{\boldmath $e$}_2)}\cap\Sigma=\{(c_j,c)\}.
\label{eq:7}
\end{eqnarray}
Then, the function $e^{-\tau/(2s_{\Sigma}(\mbox{\boldmath $\xi$}))}I(\tau ;\mbox{\boldmath $\xi$})$ is truly algebraically decaying as $\tau\rightarrow\infty$ and it holds that
\begin{equation}\label{eq:8}
\lim_{\tau\rightarrow\infty}\frac{\log\vert I(\tau ;\mbox{\boldmath $\xi$})\vert}{\tau}=\frac{1}{2s_{\Sigma}(\mbox{\boldmath $\xi$})}.
\end{equation}
\end{theorem}

Since the area between the gaps is given by
\[
W=\bigcup_{j=1}^m(\,]c_{2j-1},c_{2j}[\,\times\{c\}),
\]
the condition \eqref{eq:7} is satisfied with $\mbox{\boldmath $\xi$}=(\xi_1,\xi_2)\in\Gamma_{\epsilon}$ such that $\xi_1\in [c_{2j-1}, c_{2j}]\setminus\{(c_{2j}+c_{2j+1})/2\}$ with $j=1,\cdots, m$.
Thus, equation \eqref{eq:8} in Theorem \ref{th1} is satisfied and by definition of $s_{\Sigma}(\mbox{\boldmath $\xi$})$ we have for $\mbox{\boldmath $\xi$}\in\Gamma_{\epsilon}$ with $\xi_1\in [c_{2j-1}, c_{2j}]$ the explicit form:
\begin{equation}\label{eq:probingFunc}
s_{\Sigma}(\mbox{\boldmath $\xi$})=\left\{
\begin{array}{ll}
\displaystyle
\frac{(\xi_1-c_{2j-1})^2+(\xi_2-c)^2}{2(\xi_2-c)} & \text{if $\displaystyle c_{2j-1}\le\xi_1\le\frac{c_{2j-1}+c_{2j}}{2}$,}\\
\\
\displaystyle
\frac{(\xi_1-c_{2j})^2+(\xi_2-c)^2}{2(\xi_2-c)} & \text{if $\displaystyle \frac{c_{2j-1}+c_{2j}}{2}<\xi_1\le c_{2j}$.}
\end{array}
\right.
\end{equation}


\subsection{The probing algorithm}\label{sec:probingAlgoTheo}
The second purpose of this paper is to establish a computational probing algorithm that will indicate the tips of the cracks $\Sigma$. For this purpose the above equation \eqref{eq:probingFunc} will be the starting point for the probing procedure. Let us denote the right-hand side of \eqref{eq:8} by $\Phi(\xi_1)=\frac{1}{2}s_{\Sigma}^{-1}(\xi_1,\xi_2)$ with $\xi_2=b+\epsilon$ fixed. Then theoretically, we should see the following:
\begin{itemize}
\item There is a local minimum for $\Phi(\xi_1)$ at the point $\xi_1=\frac{c_{2j-1}+c_{2j}}{2}$; this should be the only local minimum in the interval $[c_{2j-1}, c_{2j}]$.
\item We should have 
$$
 \left. \frac{d\Phi(t)}{dt}\right|_{t=c_{2j-1}} = 0 =
   \left. \frac{d\Phi(t)}{dt}\right|_{t=c_{2j}}.
$$
\item By equation \eqref{eq:probingFunc} we should actually see that $\Phi(\xi_1)$ on $[ c_{2j-1}, c_{2j} ]$ attains a maxima at the points $\xi_1=c_{2j-1}$ and $\xi_1=c_{2j}$.
\end{itemize}

In short, our numerical approach is motivated by formula \eqref{eq:8} and the explicit form of $s_{\Sigma}(\mbox{\boldmath $\xi$})$ given in \eqref{eq:probingFunc}. 
Thus, we propose the following procedure for probing the domain from the top.
\renewcommand{\theenumi}{\roman{enumi}}
\renewcommand{\labelenumi}{{\rm (\theenumi)}}
\begin{enumerate}
\item Fix a position of $\xibf\in\Gamma_{\epsilon}$.
\item Compute $\log\vert I(\tau ;\mbox{\boldmath $\xi$})\vert$ for a selection of $\tau$.
\item Estimate the slope by linear regression and record the value.
\item Change position $\mbox{\boldmath $\xi$}$ (horizontally) and repeat (ii) until domain covered.
\end{enumerate}
Observing a change in the estimated slope of $\log\vert I(\tau ;\mbox{\boldmath $\xi$})\vert$ which may approximate $\Phi(\xi_1)$ for $(\xi_1,\xi_2)\in \Gamma_\epsilon$, we expect to find the positions of the tips of $\Sigma$.

However, note that Theorem \ref{th1} does not ensure the existence of the limit of \eqref{eq:8} if the projection of $\mbox{\boldmath $\xi$}\in\Gamma_{\epsilon}$ onto the line $x_2=c$ belongs to the set $\Sigma\setminus\{c_0,c_1,\cdots,c_{2m+1}\}$ or $\cup_{j=1}^{m}\{(c_{2j-1}+c_{2j})/2,c)\}$.
The latter set is a discrete set and hence may not pose serious problems in the computations. 
In the former case, from the previous results mentioned above we have: there exist positive numbers $C_1$ and $C_2$ such that
\begin{eqnarray}
e^{-\tau/(2s)}\vert I(\tau ;\mbox{\boldmath $\xi$})\vert\le C_1e^{-\tau C_2}
\label{eq:8-1}
\end{eqnarray}
for all $\tau \gg 1$, where $s$ is given by \eqref{eq:4} and coincides with $s_{\Sigma}(\mbox{\boldmath $\xi$})$ as mentioned above.
Here assume that the zero set $Z_{\Sigma}(\mbox{\boldmath $\xi$}):=\{\tau>0\,\vert I(\tau;\mbox{\boldmath $\xi$})=0\}$ is {\it bounded}.
Then, \eqref{eq:8-1} implies
$$\displaystyle
\limsup_{\tau\rightarrow\infty}\frac{\log\vert I(\tau ;\mbox{\boldmath $\xi$})\vert}{\tau}
\le \frac{1}{2s_{\Sigma}(\mbox{\boldmath $\xi$})}-C_2<\frac{1}{2s_{\Sigma}(\mbox{\boldmath $\xi$})}.
$$
This also suggests that the profile of $\log\vert I(\tau ;\mbox{\boldmath $\xi$})\vert/\tau$ for $\mbox{\boldmath $\xi$}$ with $(\xi_1,c)\in\Sigma\setminus\{c_0,c_1,\cdots,c_{2m+1}\}$ and with $(\xi_1,c)\in \overline{W}\setminus\cup_{j=1}^{m}\{(c_{2j-1}+c_{2j})/2,c)\}$ will pick up information about all the tips of $\Sigma$.  
The boundedness of $Z_{\Sigma}(\mbox{\boldmath $\xi$})$ for $\mbox{\boldmath $\xi$}$ with $(\xi_1,c)\in\Sigma\setminus\{c_0,c_1,\cdots,c_{2m+1}\}$ remains open.

We conclude this section with a remark about the partial boundary case, especially relevant for the application where we can not apply a current on the whole boundary $\bndry$.

\begin{remark}\label{rem1}
Assume that we know a positive number $M$ such that
\[
\sup_{\xi\in\Gamma_{\epsilon}}s_{\Sigma}(\mbox{\boldmath $\xi$})<M.
\]
Let $\delta$ be an arbitrary positive number satisfying
\[
\delta>c-(b+\epsilon-2M).
\]
Define
\begin{equation*}\label{eqn:partialIndicator}
I_{\delta}(\tau ;\mbox{\boldmath $\xi$}):=
\int_{(\partial\Omega)_{\delta}}\left(g(\mbox{\boldmath $x$})v_{\tau}(\mbox{\boldmath $x$};\mbox{\boldmath $\xi$})-\frac{\partial v_{\tau}(\mbox{\boldmath $x$};\mbox{\boldmath $\xi$})}{\partial\mbox{\boldmath $\nu$}_{\mbox{\boldmath $x$}}}u(\mbox{\boldmath $x$})\right)\; {\rm d}s_{\mbox{\boldmath $x$}},
\end{equation*}
where $(\partial\Omega)_{\delta}=\partial\Omega\cap\{\mbox{\boldmath $x$}\in\R^2\,\vert\,x_2>c-\delta\}$.
Then, it is easy to see that, as $\tau\rightarrow\infty$
\[
e^{-\tau/(2s_{\Sigma}(\mbox{\boldmath $\xi$}))}I(\tau ;\mbox{\boldmath $\xi$})
=e^{-\tau/(2s_{\Sigma}(\mbox{\boldmath $\xi$}))}I_{\delta}(\tau ;\mbox{\boldmath $\xi$})+O(\tau^{-\infty}).
\]
Thus, Theorem \ref{th1} also holds if $I(\tau ;\mbox{\boldmath $\xi$})$ is replaced with $I_{\delta}(\tau ;\mbox{\boldmath $\xi$})$.
However, if $M$ is too large, it means that there is a large connected component of $W$, then $\delta$ should be large.
In this case $(\partial\Omega)_{\delta}=\partial\Omega$.
\end{remark}

\section{Proof of the main result}\label{sec:proof}

The proof of Theorem \ref{th1} proceeds along the same line as in \cite{Ikehata2016}.  
It suffices to prove that the function
$$\begin{array}{ll}
\displaystyle
J(\tau ;\mbox{\boldmath $\xi$}):=e^{-\tau/(2s_{\Sigma}(\mbox{\boldmath $\xi$}))}I(\tau;\mbox{\boldmath $\xi$}), & \tau>0,
\end{array}
$$
is decaying {\it truly algebraically} as $\tau\rightarrow\infty$ for each fixed $\mbox{\boldmath $\xi$}\in\Gamma_\epsilon$ satisfying \eqref{eq:7}.
More precisely we show: there exist a positive number $B$ and an integer $N\ge 1$ such that
\begin{eqnarray}
\lim_{\tau\rightarrow\infty}\tau^{\frac{2N-1}{2}}\vert J(\tau ;\mbox{\boldmath $\xi$})\vert =B.
\label{eq:9}
\end{eqnarray}
For this purpose we describe two important facts.

\subsection{Preliminary facts}

We denote by $u(\,\cdot\,,c\pm 0)\in H^{1/2}(0,\,a)$ the trace of $u^{\pm}=u\vert_{\Omega^{\pm}}\in H^1(\Omega^{\pm})$ onto $]0,\,a[\,\times\{c\}$, respectively.
First of all the proof is based on the following representation formula of the indicator function which can be proved by using \eqref{eq:1} and integration by parts.

\begin{proposition}[Proposition 1 in \cite{Ikehata2016}]\label{prop2}
The formula
\begin{eqnarray*}
I(\tau ; \mbox{\boldmath $\xi$})=-\int_{\Sigma} (u^{+}(\mbox{\boldmath $x$})-u^{-}(\mbox{\boldmath $x$})) \frac{\partial v_{\tau}(\mbox{\boldmath $x$}; \mbox{\boldmath $\xi$})}{\partial x_{2}} \; {\rm d}s_{\mbox{\boldmath $x$}}
\end{eqnarray*}
is valid.
\end{proposition}

Note that as described in Section \ref{sec:theoIntro} we have already proven Theorem \ref{th1} in the special case when the projection of $\mbox{\boldmath $\xi$}\in\Gamma_{\epsilon}$ onto the line $x_2=c$ coincides with a tip of $\Sigma$ in $\Omega$.
Thus, the case to be considered is: the projection of $\mbox{\boldmath $\xi$}$ onto the line $x_2=c$ belongs to some $\left(]c_{j}, c_{j+1}[\setminus\{(c_j+c_{j+1})/2\}\right)\times\{c\}=]c_j, (c_j+c_{j+1})/2[\times\{c\}\,\cup\, ](c_j+c_{j+1})/2,c_{j+1}[\times\{c\}$ with an {\it odd} number $j$.
Here we consider only the case when the projection of $\mbox{\boldmath $\xi$}$ onto the line $x_2=c$ belongs to $]c_j, (c_j+c_{j+1})/2[\times\{c\}$.

The next important ingredient of the proof is a {\it convergent expansion formula} of $u$ around $(c_j,c)$ as stated below.
Choose a positive number $\eta_0$ in such a way that $\eta_0<\min_{j=1,\cdots, 2m+1}(c_j-c_{j-1})$ and $\eta_0<\min(b-c, c)$.
We choose a polar coordinates system centered at $(c_j,c)$ as done in \cite{Ikehata2016} and illustrated in Figure \ref{fig00-3}.

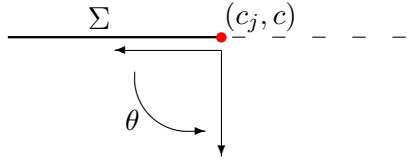
\begin{figure}[h!]
\centering
\begin{picture}(150,60)
{\linethickness{0.3mm}
\put(0,45){\Large{\line(1,0){80}}}
}
\put(30,48){\large{$\Sigma$}}
\multiput(9,45)(15,0){10}{\line(1,0){5}}

\put(80,40){\vector(-1,0){40}}
\put(80,40){\vector(0,-1){40}}
\put(81,50){$(c_j,c)$}
\put(77.5,42.5){\small{\tredbf{$\bullet$}}}
\put(70,32){\oval(45,45)[bl]}
\put(73,9.5){\small{\vector(1,0){1}}}
\put(44,10){$\theta$}
\end{picture}
\caption{\label{fig00-3}Illustration of the polar coordinates system used.}
\end{figure}

Set $\mbox{\boldmath $x$}:=(c_j-r\cos{\theta}, c-r\sin{\theta})$ for $0<r<\eta_0$ and $0<\theta<2\pi$ and define
$$\begin{array}{lll}
\displaystyle
u(r,\theta)=u(c_j-r\cos{\theta},c-r\sin{\theta}), & \displaystyle
0<r<\eta_0, & \displaystyle 0<\theta<2\pi.
\end{array}
$$
Hence, we have
$$\begin{array}{ll}
\displaystyle
u^+(\mbox{\boldmath $x$})=u(r,\theta), & \displaystyle \pi <\theta <2\pi,
\end{array}
$$
and
$$\begin{array}{ll}
\displaystyle
u^-(\mbox{\boldmath $x$})=u(r,\theta), & \displaystyle 0<\theta<\pi .
\end{array}
$$

\begin{proposition}[Proposition 2 in \cite{Ikehata2016}]\label{prop2-1}
Fix $ 0<\eta <\frac{\eta_{0}}{2} $.
There exist a real number $ M $ and a sequence $ \{A^{(j)}_{k}\} $ of real numbers such that
\[
u(r,\theta)-M =\sum_{k=1}^{\infty}r^{\frac{k}{2}}A^{(j)}_{k}\cos{\frac{k\theta}{2}},\quad 0<r<2\eta ,\quad 0<\theta<2\pi .
\]
The series is absolutely convergent in $ H^{1}(B_{\eta}((c_j ,c))\cap \Omega^{+})$ and $ H^{1}(B_{\eta}((c_j ,c))\cap \Omega^{-})$, and uniformly in $ B_{2\eta}((c_j ,c)) $.
Moreover, for each $ n=1,2,\cdots $ the following estimate is valid uniformly for $ 0<r<\eta $:
\begin{align*}
&\left\vert u(c_j -r, c-0)-M -\sum_{k=1}^{n}r^{\frac{k}{2}}A^{(j)}_{k}\right\vert  \ \\
&\ \hspace{2 cm} +  \left\vert u(c_j -r, c+0)-M -\sum_{k=1}^{n}r^{\frac{k}{2}}A^{(j)}_{k}(-1)^{k}\right\vert \leq K_{n} r^{\frac{n+1}{2}},
\end{align*}
where $ K_{n} $ is a positive constant depending on $ n $.
\end{proposition}
The proof is an adaptation of the argument described in \cite{Grisvard1985} which is based on an eigenfunction expansion associated with the operator $-\frac{d^2}{dx^2}$ with homogeneous Neumann boundary condition.
See also \cite{Ikehata1999} and Appendix of \cite{Ikehata2003b}.

\subsection{Proof of Equation \eqref{eq:9}} \label{sec:proofLim}

In what follows we set for simplicity of description
\[
s_0:=s_{\Sigma}(\mbox{\boldmath $\xi$}).
\]

From Proposition \ref{prop2-1} we have, for each $ n=1,2,\cdots $,
\begin{equation}
u(c_j -r,c+0)-u(c_j -r,c-0)=-2\sum_{k=1}^{n}r^{\frac{2k-1}{2}}A^{(j)}_{2k-1}+O\left( r^{\frac{2n+1}{2}}\right).
\label{eq:2-1}
\end{equation}
Next, choose a positive number $\delta$ in such a way that
$$\displaystyle
\overline{B_{s_0+\delta}(\mbox{\boldmath $\xi$}-s_0\mbox{\boldmath $e$}_2)}\cap\Sigma \subset [c_{j-1},c_j]\times\{c\}
$$
and
$$
\displaystyle
\eta_{\delta}:=\sqrt{(s_0+\delta)^2-s_0^2}<\eta .
$$
Set 
$$\displaystyle
\eta_{\delta}'=\sqrt{(s_{0}+\delta)^2-(\xi_{2}-s_{0}-c)^2}-\vert\xi_{1}-c_j\vert.
$$
Then one sees $\eta_{\delta}>\eta'_{\delta}>0$.

Using Proposition \ref{prop2}, we divide $J(\tau ;\mbox{\boldmath $\xi$})$ into two parts:
\begin{eqnarray*}
J(\tau ;\mbox{\boldmath $\xi$})=J_1(\tau)+J_2(\tau),
\end{eqnarray*}
where
$$
\begin{array}{ll}
\displaystyle
J_j(\tau):=-e^{-\frac{\tau}{2s_0}}\int_{\Sigma_j}
(u^+(\mbox{\boldmath $x$})-u^{-}(\mbox{\boldmath $x$}))\frac{\partial v_{\tau}(\mbox{\boldmath $x$}; \mbox{\boldmath $\xi$})}{\partial x_2}\; {\rm d}s_{\mbox{\boldmath $x$}}, & j=1,2,
\end{array}
$$
and
$$\left\{
\begin{array}{l}
\displaystyle
\Sigma_1 := \Sigma\setminus B_{s_0+\delta}(\mbox{\boldmath $\xi$}-s_0\mbox{\boldmath $e$}_2),\\
\\
\displaystyle
\Sigma_2 := \Sigma\cap B_{s_0+\delta}(\mbox{\boldmath $\xi$}-s_0\mbox{\boldmath $e$}_2).
\end{array}
\right.
$$

First we give an upper bound of $J_1(\tau)$.
Since for all $\mbox{\boldmath $x$}\in\Sigma_1$ we have $\vert \mbox{\boldmath $\xi$}-(\mbox{\boldmath $\xi$}-s_0\mbox{\boldmath $e$}_2)\vert\ge s_0+\delta$, one gets
\[
\frac{\mbox{\boldmath $x$}-\mbox{\boldmath $\xi$}}{\vert \mbox{\boldmath $x$}-\mbox{\boldmath $\xi$}\vert^{2}}\cdot\mbox{\boldmath $e$}_{2}+\frac{1}{2s_{0}}=\frac{\vert \mbox{\boldmath $x$}-(\mbox{\boldmath $\xi$}-s_{0}\mbox{\boldmath $e$}_{2})\vert^{2}-s_{0}^{2}}{2s_{0}\vert \mbox{\boldmath $x$}-\mbox{\boldmath $\xi$}\vert^{2}}
\ge C_{3}\eta_{\delta}^{2},
\]
for some positive constant $C_3$ being independent of $\delta$.
This yields
$$\begin{array}{ll}
\displaystyle
e^{-\frac{\tau}{2s_{0}}}\left\vert\frac{\partial v_{\tau}(\mbox{\boldmath $x$};\mbox{\boldmath $\xi$})}{\partial x_2}\right\vert \le C_{4}\tau e^{-C_{3}\eta_{\delta}^{2}\tau},
& 
\displaystyle
\mbox{\boldmath $x$}\in\Sigma_{1},
\end{array}
$$
where $C_4$ is a positive constant.
From a combination of this and \eqref{eq:2-1} we obtain, as $\tau\rightarrow\infty$
\begin{eqnarray}
J_1(\tau)=O\left( \tau e^{-C_{3}\eta_{\delta}^2\tau}\right) .
\label{eq:2-2}
\end{eqnarray}

One of the key points of the proof of \eqref{eq:9} is the following asymptotic formula of $J_2(\tau)$ as $\tau\rightarrow\infty$:
\begin{eqnarray}
J_2(\tau)
=-2\tau e^{-\frac{\tau}{2s_0}}\sum_{k=1}^n A_{2k-1}^{(j)}I_k(\tau)+O\left( \tau^{-\frac{n+1}{2}}\tau^{\frac{3}{4}}\right),
\label{eq:2-3}
\end{eqnarray}
where
\[
I_k(\tau):=\int_0^{\eta_{\delta}'}
\frac{r^{\frac{2k-1}{2}}}{(r-s_0\overline{z_{\alpha}})^2}
\exp\left( \frac{i\tau}{r-s_0\overline{z_{\alpha}}}\right)\; {\rm d}r,
\]
\[
z_{\alpha}:=-\left( e^{-\frac{\pi}{2}i}+ie^{-(\frac{\pi}{2}+\alpha)i}\right)=-\cos\alpha+i(1+\sin\alpha)
\]
and $\alpha\in\,]-\frac{\pi}{2},\,\frac{\pi}{2}[$ is the unique solution of the equation
\[
e^{i\alpha}=\frac{\xi_1-c_j}{s_0}+i\frac{\xi_2-s_0-c}{s_0}.
\]
Note that $(c_j,c)\in\partial B_{s_0}(\mbox{\boldmath $\xi$}-s_0\mbox{\boldmath $e$}_2)$ and $\xi_1>c_j$ ensures the unique existence.

Equation \eqref{eq:2-3} is proved as follows.
First from the definition of $\eta_{\delta}'$ we have
$$\displaystyle
\Sigma_2=\, ]c_j-\eta_{\delta}',\,c_j[\,\times\{c\}.
$$
This together with the change of the variable $r=c_j-t$, we have
$$\displaystyle
J_2(\tau)
=-e^{-\frac{\tau}{2s_0}}
\int_0^{\eta_{\delta}'}
(u(c_j-r,c+0)-u(c_j-r,c-0))\frac{\partial v_{\tau}}{\partial x_2}((c_j-r,c);\mbox{\boldmath $\xi$})\; {\rm d}r.
$$
Then, \eqref{eq:2-1} gives
\begin{equation}\label{eq:2-4}
\begin{array}{l}
\displaystyle
\,\,\,\,\,\,
\left\vert J_2(\tau)+2\tau e^{-\frac{\tau}{2s_0}}\sum_{k=1}^nA_{2k-1}^{(j)}I_k(\tau)\right\vert\\\\
\displaystyle
\le \tilde{C_n}e^{-\frac{\tau}{2s_0}}\int_0^{\eta'_{\delta}}r^{\frac{2n+1}{2}}\left\vert\frac{\partial v_{\tau}}{\partial x_2}((c_j-r,c);\xi)\right\vert dr.
\end{array}
\end{equation}
Here we see that, for $\mbox{\boldmath $x$}=(c_j-r,c)$ we have
\[
\frac{\mbox{\boldmath $x$}-\mbox{\boldmath $\xi$}}{\vert \mbox{\boldmath $x$}-\mbox{\boldmath $\xi$}\vert^{2}}\cdot(\mbox{\boldmath $e$}_{2}+i\mbox{\boldmath $e$}_{1})
=\frac{-i}{r+(\xi_{1}-c_j )+i(\xi_{2}-c)},
\]
and thus
\begin{eqnarray*}
\frac{\partial}{\partial x_{2}}\left(\frac{\mbox{\boldmath $x$}-\mbox{\boldmath $\xi$}}{\vert \mbox{\boldmath $x$}-\mbox{\boldmath $\xi$}\vert^{2}}\cdot(\mbox{\boldmath $e$}_{2}+i\mbox{\boldmath $e$}_{1})\right)
= \frac{1}{\left( r+(\xi_{1}-c_j )+i(\xi_{2}-c) \right)^{2}}.
\end{eqnarray*}
This together with the choice of $\alpha$ yields the expression
\begin{eqnarray*}
\frac{\partial v_{\tau}}{\partial x_{2}}((c_j -r,c);\mbox{\boldmath $\xi$})
= -\frac{\tau}{(r-s_0\overline{z_{\alpha}})^2}\,\exp\left( \frac{i\tau}{r-s_0\overline{z_{\alpha}}}\right).
\end{eqnarray*}
Applying this and the following fact to the right-hand side on \eqref{eq:2-4}, we obtain \eqref{eq:2-3}.

\begin{lemma}\label{lem1}
Let $ n=1,2,\cdots $.
We have, as $\tau\rightarrow\infty$
\[
\tau e^{-\frac{\tau}{2s_0}}
\int_0^{\eta_{\delta}'}
r^{\frac{2n+1}{2}}
\left\vert\frac{1}{(r-s_0\overline{z_{\alpha}})^2}\,\exp\left(
\frac{i\tau}{r-s_0\overline{z_{\alpha}}}\right)\right\vert\; {\rm d}r
=O\left( \tau^{-\frac{n+1}{2}} \tau^{1-\frac{1}{4}}\right).
\]
\end{lemma}

\begin{proof}
Noting $2s_0\cos\alpha>0$ in the case when $j$ is odd, we have, for $0<r<\eta_{\delta}'$ 
\begin{eqnarray*}
e^{-\frac{\tau}{2s_0}}
\left\vert\frac{1}{(r-s_0\overline{z_{\alpha}})^2}\,\exp\left(
\frac{i\tau}{r-s_0\overline{z_{\alpha}}}\right)\right\vert
&=& \frac{e^{-\tau/(2s_0)}e^{\tau s_0(1+\sin\alpha)/\vert r-s_0\overline{z_{\alpha}}\vert^2}}{\vert r-s_0\overline{z_{\alpha}}\vert^2}\\
&=& \frac{e^{-\tau r(r+2s_0\cos\alpha)/(2s_0\vert r-s_0\overline{z_{\alpha}}\vert^2)}}{\vert r-s_0\overline{z_{\alpha}}\vert^2}\\
&\le & \frac{e^{-\tau r(r+2s_0\cos\alpha)/(2s_0\vert\eta_{\delta}'-s_0\overline{z_{\alpha}}\vert^2)}}{2s_0^2(1+\sin\alpha)}.
\end{eqnarray*}
Thus, there exist positive constants $C_5$ and $C_6$ such that
\begin{equation*}
\begin{array}{l}
\displaystyle
\,\,\,\,\,\,
\tau e^{-\frac{\tau}{2s_0}}\int_0^{\eta'_{\delta}}r^{\frac{2n+1}{2}}\left\vert\frac{1}{(r-s_0\overline{z_{\alpha}})^2}
\exp\left(\frac{i\tau}{r-s_0\overline{z_{\alpha}}}\right)\right\vert\,dr
\\
\\
\displaystyle
\le
C_5\int_0^{\eta'_{\delta}}\tau r^{\frac{2n+1}{2}}e^{-C_6\tau r^2}\,dr.
\end{array}
\end{equation*}
Since a change of variable yields
\[
\int_{0}^{\eta'_{\delta}}\tau r^{(2n+1)/2}e^{-C_{6}\tau r^{2}}\; {\rm d}r
=O\left( \tau^{-\frac{n+1}{2}}\tau^{1-\frac{1}{4}}\right),
\]
we obtain the desired estimate.\\
\end{proof}

The following fact is another key point for the proof of \eqref{eq:9}.

\begin{lemma}\label{lem1-2}
Let $g\, (\not=0)$ satisfy the conditions {\rm A1.)} or {\rm A2.)}.
Then, there exists an integer $n\ge 1$ such that $A_{2n-1}^{(j)}\not=0$.
\end{lemma}

\begin{proof}
We basically follow the flow of the argument done in Lemma 2 in \cite{Ikehata2016}.
However, since the condition A2.) is new, we present the proof precisely.
We employ a contradiction argument.
Assume that the conclusion is not true.
Then, from the assumption we see that the expression in Proposition \ref{prop2-1} becomes
$$\displaystyle
u(r,\theta)=M+\sum_{n=1}^{\infty} r^nA_{2n}^{(j)}\cos n\theta.
$$
This right-hand side gives $\frac{\partial u}{\partial x_2}(c_j+r,c)=0$ for $0<r\ll 1$.
Then, we see that the harmonic function $u^{+}(x_1,x_2)-u^{-}(x_1,2c-x_2)$ with $0<x_1<a$ and $c<x_2<\min(c, b-c)$ has vanishing Cauchy data on $c_j<x_1<c_j +r $ and $x_2=c$.
Thus the uniqueness of the Cauchy problem for the Laplace equation enables us to conclude that $u^{+}(x_1,x_2)-u^{-}(x_1,2c-x_2)=0$ with $0<x_1<a$ and $c<x_2<\min(c,b-c)$.
This is a reflection argument in \cite{Alessandrini1997}.
Then taking the Neumann derivative on $W$ from $x_2>c$, one concludes
$\frac{\partial u}{\partial x_2}(\mbox{\boldmath $x$})=0$ for all $\mbox{\boldmath $x$}\in W$.
Combining this with the boundary condition on $\Sigma$, we see that all the tips of $\Sigma$ are removable singularities of $u^+$ and $u^{-}$.
Moreover, from the assumption on the support of $g$ in A1.) or A2.) one concludes all the corner points of $\Omega^{+}$ and $\Omega^{-}$ are also removable singularities of $u^+$ and $u^{-}$, respectively.
Therefore one gets
$$\displaystyle
\int_{\partial\Omega^{\pm}}\frac{\partial u}{\partial\mbox{\boldmath $\nu$}}\; {\rm d}s_{\mbox{\boldmath $x$}}=0.
$$
This yields 
$$\displaystyle
\int_{\partial\Omega\cap\{x_2>c\}}g\; {\rm d}s_{\mbox{\boldmath $x$}}=0
$$
and
$$\displaystyle
\int_{\partial\Omega\cap\{x_2<c\}}g\; {\rm d}s_{\mbox{\boldmath $x$}}=0.
$$
This tells us that $A_{2n-1}^{(j)}=0$ for all $n=1,2,\cdots$ never occur if $g$ satisfies A2.).
Moreover, if $g$ satisfies A1.), then one gets $\frac{\partial u}{\partial\mbox{\boldmath $\nu$}}=0$ on $\partial\Omega^{-}$ and thus $u^-$ has to be a constant in $\Omega^{-}$.
Passing through $x_2=c$, one gets $u^+$ also has to be the same constant.
Then, one gets $g=0$.
This contradicts $g\not=0$.\\
\end{proof}

Now let us assume that A1.) or A2.) are satisfied with $g\not=0$.
By Lemma \ref{lem1-2}, we can find the smallest $n\ge 1$ such that $A_{2n-1}^{(j)}\not=0$.
Set
$$\displaystyle
N:=\min\{n\ge 1\,\vert A_{2n-1}^{(j)}\not=0\}.
$$
From \eqref{eq:2-2} and \eqref{eq:2-3} one can write
\begin{equation}\label{eq:2-5}
\begin{split}
J(\tau ;\mbox{\boldmath $\xi$}) = &2A_{2N-1}^{(j)}
\tau^{1-\frac{2N+1}{2}}e^{-\frac{\tau}{2s_0}}\tau^{\frac{2N+1}{2}}I_N(\tau)  \\
&+\sum_{n=N+1}^{2N}2A_{2n-1}^{(j)}
\tau^{1-\frac{2n+1}{2}}e^{-\frac{\tau}{2s_0}}\tau^{\frac{2n+1}{2}}I_n(\tau)
+O\left( \tau^{-\frac{2N+1}{2}}\tau^{1-\frac{1}{4}}\right).
\end{split}
\end{equation}

Here we state a new key lemma which has not been covered in \cite{Ikehata2016} and whose proof is given in the next section.

\begin{lemma}\label{lem1-3}
Let $n=1,2, \cdots$ and $\alpha \in\, \bigr] -\frac{\pi}{2},\frac{\pi}{2} \bigl[$.
We have
\[
\begin{array}{l}
\displaystyle
\,\,\,\,\,\,
\lim_{\tau\rightarrow\infty}
\tau^{\frac{2n+1}{2}}
e^{-\frac{\tau}{2s_0}}
e^{-\frac{i\tau\cos\alpha}{2s_0(1+\sin\alpha)}}I_n(\tau)
\\
\\
\displaystyle
=-is_0^{2n-1}2^{\frac{2n-1}{2}}(1+\sin\alpha)^{\frac{2n-1}{2}}
e^{i\frac{(2n-1)\alpha}{2}}
\Gamma\left(\frac{2n+1}{2}\right).
\end{array}
\]
\end{lemma}

Thus, applying Lemma \ref{lem1-3} to each $n=N,\cdots,2N$, it follows from \eqref{eq:2-5} that
\[
\tau^{\frac{2N+1}{2}-1}J(\tau ;\mbox{\boldmath $\xi$})
=2A_{2N-1}^{(j)}e^{-\frac{\tau}{2s_0}}\tau^{\frac{2N+1}{2}}I_N(\tau)
+O\left(\tau^{-\frac{1}{4}}\right)
\]
and finally we obtain \eqref{eq:9} with
\[
B=\vert A_{2N-1}^{(j)}\vert
s_0^{2N-1}2^{\frac{2N+1}{2}}
(1+\sin\alpha)^{\frac{2N-1}{2}}\Gamma\left(\frac{2N+1}{2}\right).
\]
Thus, to complete the proof of Theorem \ref{th1} it suffices to prove Lemma \ref{lem1-3}.

\subsection{Proof of Lemma \ref{lem1-3}}

For the proof of Lemma \ref{lem1-3} we also treat only the case when $j$ is odd, however, the procedure of the proof is same in the case $j$ is even.
We apply the {\it method of steepest descent} to the integral $I_n(\tau)$.

\subsubsection{Step 1}

Let $z=\eta e^{i\gamma}$ with $\gamma\in\R$ and $\eta>0$.
We have
\begin{eqnarray*}
\left\{
\begin{array}{l}
\vspace{0.3cm}
\displaystyle
z-s_0\overline{z_{\alpha}} = \eta\cos\gamma+s_0\cos\alpha+i\{\eta\sin\gamma+s_0(1+\sin\alpha)\},\\
\displaystyle
\vert z-s_0\overline{z_{\alpha}}\vert^2 = \eta^2+2\eta s_0\sin\gamma+2\eta s_0\cos(\gamma-\alpha)+2s_0^2(1+\sin\alpha).
\end{array}
\right.
\end{eqnarray*}
This gives
\[
{\bf Re}\,\frac{i}{z-s_0\overline{z_{\alpha}}}
=\frac{\eta\sin\gamma+s_0(1+\sin\alpha)}{\vert z-s_0\overline{z_{\alpha}}\vert^2}
\]
and thus
\[
{\bf Re}\,\frac{i}{z-s_0\overline{z_{\alpha}}}-\frac{1}{2s_0}
=-\frac{\eta(\eta+2s_0\cos(\gamma-\alpha))}{2s_0\vert z-s_0\overline{z_{\alpha}}\vert^2}.
\]
Therefore if $\vert\gamma-\alpha\vert\le \frac{\pi}{2}$, we have
\begin{eqnarray}
{\bf Re}\,\frac{i}{z-s_0\overline{z_{\alpha}}}-\frac{1}{2s_0}
\le -\frac{\eta^2}{2s_0(\eta+2s_0)^2}.
\label{eq:2-6}
\end{eqnarray}

\subsubsection{Step 2}

We seek the set of all $z$ such that
\[
{\bf Im}\,\left(\frac{i}{z-s_0\overline{z_{\alpha}}}\right)
={\bf Im}\,\left.\left(\frac{i}{z-s_0\overline{z_{\alpha}}}\right)\right|_{z=0}.
\]
A simple calculation gives
$$\begin{array}{ll}
\displaystyle
z=\frac{s_0(1+\sin\alpha)}{\cos\alpha}(e^{i\varphi}-ie^{i\alpha}), & \varphi\in\R.
\end{array}
$$
Then, we have also
\begin{eqnarray}
\left\{
\begin{array}{l}
\vspace{0.3cm}
\displaystyle
z-s_0\overline{z_{\alpha}}
=\frac{s_0(1+\sin\alpha)}{\cos\alpha}
(e^{i\varphi}+1),\\
\displaystyle
\frac{i}{z-s_0\overline{z_{\alpha}}}
=\frac{\cos\alpha}{s_0(1+\sin\alpha)}
\left(\frac{\sin\varphi}{2(1+\cos\varphi)}+\frac{i}{2}\right).
\end{array}
\right.
\label{eq:2-7}
\end{eqnarray}

\subsubsection{Step 3}

Choose $\eta$ in Proposition \ref{prop2-1} in such a way that
\[
0<\eta<\frac{s_0(1+\sin\alpha)}{\cos\alpha}
\]
and the corresponding $\eta'_{\delta}$ in Subsection 2.2.
Define the curve
$$\begin{array}{ll}
\displaystyle
C: z=\frac{s_0(1+\sin\alpha)}{\cos\alpha}(e^{i\varphi}-ie^{i\alpha}), & 
\displaystyle
\alpha-\frac{\pi}{2}\le\varphi\le\alpha+\frac{\pi}{2}.
\end{array}
$$

This is an arc of the circle passing through $z=0$ at $\varphi=\alpha+\frac{\pi}{2}$ with radius $\frac{s_0(1+\sin\alpha)}{\cos\alpha}$ centered at
$P=\left(\frac{s_0\sin\alpha(1+\sin\alpha)}{\cos\alpha}, -s_0(1+\sin\alpha)\right)$.
Now changing the contour in $I_n(\tau)$ with the help of the Cauchy integral formula together with \eqref{eq:2-6} and \eqref{eq:2-7}, see Figure \ref{fig10} for an illustration, one has
\begin{equation} \label{eq:2-8}
\begin{split}
e^{-\frac{\tau}{2s_0}}I_n(\tau) &= e^{-\frac{\tau}{2s_0}}\int_{(-C)\cap B_{\eta_{\delta}'}(0)}
\frac{z^{\frac{2n-1}{2}}}
{(z-s_0\overline{z_{\alpha}})^2} e^{i\tau\frac{i}{z-s_0\overline{z_{\alpha}}}}\; {\rm d}z +O\left(\tau^{-\infty}\right) \\
&= -i\left( \frac{s_{0}(1+\sin{\alpha})}{\cos{\alpha}} \right)^{n-\frac{3}{2}}e^{-\frac{\tau}{2s_0}}e^{\frac{i\tau\cos{\alpha}}{2s_{0}(1+\sin{\alpha})}}\tilde{I}_{n}(\tau )+O\left( \tau^{-\infty}\right),
\end{split}
\end{equation}
where
\begin{equation} \label{eq:2-9}
\tilde{I}_{n}(\tau ) := \int_{\alpha+\frac{\pi}{2}-\epsilon (\eta'_{\delta})}^{\alpha+\frac{\pi}{2}}\frac{e^{i\varphi}(e^{i\varphi}-ie^{i\alpha})^{\frac{2n-1}{2}}}{(1+e^{i\varphi})^2} {\rm exp} \left\{ \frac{\tau\cos{\alpha}}{s_{0}(1+\sin{\alpha})}\cdot\frac{\sin{\varphi}}{2(1+\cos{\varphi})} \right\}\; {\rm d}\varphi  
\end{equation}
and $0<\epsilon (\eta'_{\delta})\ll 1$ satisfies
$$
\cos{\epsilon (\eta'_{\delta})}=1-\frac{(\eta'_{\delta})^2\cos^{2}{\alpha}}{2s_{0}^{2}(1+\sin{\alpha})^{2}}.
$$

\begin{figure}[ht!]
\centering
\begin{picture}(200,200)
\includegraphics[width=200pt]{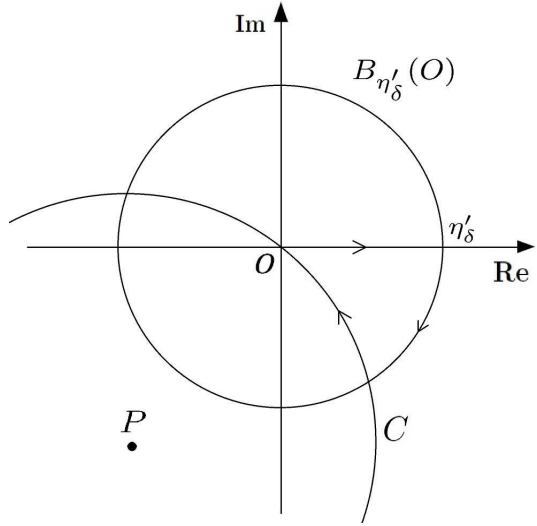}
\end{picture}
\caption{\label{fig10}An example of the contour in the case $-\frac{\pi}{2}<\alpha <0$}
\end{figure}

Here we make use of the change of variable $\varphi$ in the right-hand side on \eqref{eq:2-9} as follows:
\[
\varphi =\alpha +\frac{\pi}{2} -x,\quad 0\leq x \leq \epsilon (\eta'_{\delta})\ll 1.
\]
Using the expressions
$$\left\{
\begin{array}{cc}
\displaystyle
e^{i\varphi}=ie^{i(\alpha -x)}, &
\displaystyle
e^{i\varphi}-ie^{i\alpha}=2\sin{\frac{x}{2}}e^{i\left( \alpha -\frac{x}{2} \right)}, \\
\\
\displaystyle
(1+e^{i\varphi})^2=4i\sin^{2}{\frac{1}{2}\left( \frac{\pi}{2}-\alpha -x \right)}e^{i(\alpha -x)},
&
\displaystyle 
\frac{\sin{\varphi}}{1+\cos{\varphi}}=\frac{\cos{(\alpha -x)}}{1-\sin{(\alpha -x)}},
\end{array}
\right.
$$
one can rewrite
$$\displaystyle
\tilde{I}_{n}(\tau ) = \int^{\epsilon (\eta'_{\delta})}_{0} \frac{\left( 2\sin{\frac{x}{2}}\right)^{n-\frac{1}{2}} e^{i\left( n-\frac{1}{2}\right)\left( \alpha -\frac{x}{2} \right)}}{4\sin^{2}{\frac{1}{2}\left( \frac{\pi}{2}-\alpha -x \right)}} e^{\tau f(x)} \; {\rm d}x,
$$
where
$$\displaystyle
f(x):=\frac{\cos{\alpha}}{s_{0}(1+\sin{\alpha})}\cdot\frac{\cos{(\alpha -x)}}{2(1-\sin{(\alpha -x)})}.
$$
Note that $f(0)=\frac{1}{2s_{0}}$ and also $\displaystyle f'(0)=-\frac{1}{2s_{0}\cos{\alpha}}<0$ for all $\alpha \in\, \bigr] -\frac{\pi}{2},\frac{\pi}{2} \bigl[ $.
Thus, from Watson's lemma (cf. \cite{Olver1997}) we have, as $\tau\rightarrow\infty$
\begin{eqnarray*}
e^{-\frac{\tau}{2s_0}}\tilde{I}_{n}(\tau ) & = & 
\int^{\epsilon (\eta'_{\delta})}_{0} \frac{\left( 2\sin{\frac{x}{2}}\right)^{n-\frac{1}{2}} e^{i\left( n-\frac{1}{2}\right)\left( \alpha -\frac{x}{2} \right)}}{4\sin^{2}{\frac{1}{2}\left( \frac{\pi}{2}-\alpha -x \right)}} e^{\tau (f(x)-f(0))}\; {\rm d}x\\
&&\;\\
&\sim & \int^{\epsilon (\eta'_{\delta})}_{0} \frac{x^{n-\frac{1}{2}} e^{i\left( n-\frac{1}{2}\right)\alpha}}{4\sin^{2}{\frac{1}{2}\left( \frac{\pi}{2}-\alpha \right)}} e^{\tau f'(0)x} \; {\rm d}x\\
&&\;\\
&=& \frac{e^{i\left( n-\frac{1}{2}\right)\alpha}}{4\sin^{2}{\frac{1}{2}\left( \frac{\pi}{2}-\alpha \right)}}\int^{\epsilon (\eta'_{\delta})}_{0}x^{n-\frac{1}{2}} e^{-\tau |f'(0)|x} \; {\rm d}x\\
&&\;\\
&=& \frac{e^{i\left( n-\frac{1}{2}\right)\alpha}}{4\sin^{2}{\frac{1}{2}\left( \frac{\pi}{2}-\alpha \right)}}\left( \frac{1}{|f'(0)|}\right)^{n+\frac{1}{2}}\frac{1}{\tau^{n+\frac{1}{2}}}\int^{\tau |f'(0)|\epsilon (\eta'_{\delta})}_{0}y^{n-\frac{1}{2}} e^{-y} \; {\rm d}y\\
&&\;\\
&\sim & \frac{e^{i\left( n-\frac{1}{2}\right)\alpha}}{4\sin^{2}{\frac{1}{2}\left( \frac{\pi}{2}-\alpha \right)}}\left( \frac{1}{|f'(0)|}\right)^{n+\frac{1}{2}}\frac{1}{\tau^{n+\frac{1}{2}}}\Gamma\left( n+\frac{1}{2}\right)\\
&&\;\\
&= & \frac{e^{i\left( n-\frac{1}{2}\right)\alpha} (2s_{0}\cos{\alpha})^{n+\frac{1}{2}}}{2(1-\sin{\alpha})}\Gamma\left( n+\frac{1}{2}\right)\frac{1}{\tau^{n+\frac{1}{2}}}.
\end{eqnarray*}
Now a combination of this and \eqref{eq:2-8} yields, as $\tau\rightarrow\infty$
$$\displaystyle
e^{-\frac{\tau}{2s_{0}}}I_{n}(\tau ) \sim -i s_{0}^{2n-1}2^{n-\frac{1}{2}}(1+\sin{\alpha})^{n-\frac{1}{2}} e^{i\left( n-\frac{1}{2}\right)\alpha} 
e^{\frac{i\tau\cos{\alpha}}{2s_{0}(1+\sin{\alpha})}}\Gamma\left( n+\frac{1}{2}\right)\frac{1}{\tau^{n+\frac{1}{2}}},
$$
which directly leads to Lemma \ref{lem1-3}.\\

\section{Numerical considerations}\label{sec:numericalConsider}
A computational implementation of the enclosure method is not straight-forward due to the limit in \eqref{eq:8}, i.e. as $\tau\to\infty$. Naturally this limit can not be computationally evaluated, but this limitation can be overcome by rather observing the asymptotic behavior of the indicator function $\log|I(\tau;\xibf)|$ that already shows for small values of $\tau$. This modification has been introduced first by the two studies \cite{Ikehata2000a} and \cite{Bruhl2000}. Specifically, this is done by evaluating the indicator function for a set of values $\tau$ and estimating the slope of $\log|I(\tau;\xibf)|$ will then give a robust estimate of the desired support function. In the following we will discuss the specific modifications needed for the application by the Kelvin transformed probing function.



\subsection{Numerical considerations for the Kelvin transform}\label{sec:NumericalKelvin}
We will review here the essential steps needed to implement the enclosure method with Kelvin transformed probing functions. This transformation is necessary to detect the crack tips inside the domain. 

Before we can evaluate the indicator function, we need to obtain a pair of Cauchy data for equation \eqref{eq:0}, which can be obtained by a FEM simulation with fixed Neumann data $g$. We will model the crack as a cavity with conductivity zero and small width, as similarly done in \cite{Ikehata2002c,Ikehata2008}. Having obtained a pair of Cauchy data $(g,u)$ we can now start the probing procedure. For the convenience of the reader, we will repeat the essential functions and equations for the computations here.
First of all, we need the special solution given in equation \eqref{eq:0} for $\mbox{\boldmath $\xi$}\in\R^2\setminus\overline{\Omega}$ by
\begin{equation*}
v_{\tau}(\mbox{\boldmath $x$};\mbox{\boldmath $\xi$})=\exp\left\{-\tau\frac{\mbox{\boldmath $x$}-\mbox{\boldmath $\xi$}}{\vert \mbox{\boldmath $x$}-\mbox{\boldmath $\xi$}\vert^2}\cdot(\mbox{\boldmath $e$}_2+i\mbox{\boldmath $e$}_1)\right\},\quad \mbox{\boldmath $x$}\in\R^2\setminus\{\mbox{\boldmath $\xi$}\}.
\end{equation*}
Then the associated indicator function is given by
\begin{equation}\label{eqn:kelvinIndicator}
I(\tau ;\mbox{\boldmath $\xi$}):=\int_{\partial\Omega}
\left(g(\mbox{\boldmath $x$})v_{\tau}(\mbox{\boldmath $x$};\mbox{\boldmath $\xi$})-u(\mbox{\boldmath $x$})\frac{\partial v_{\tau}(\mbox{\boldmath $x$};\mbox{\boldmath $\xi$})}{\partial\mbox{\boldmath $\nu$}_{\mbox{\boldmath $x$}}}\right)\; {\rm d}s_{\mbox{\boldmath $x$}},
\end{equation}
or in the partial boundary case by $I_{\delta}(\tau ;\mbox{\boldmath $\xi$})$ as discussed in Remark \ref{rem1}.

Let us fix $\mbox{\boldmath $\xi$}\in \R^2\setminus\overline\Omega$ which is the ``north pole'' of the probing disc
$B_s(\mbox{\boldmath $\xi$}-s\mbox{\boldmath $e$}_2)$. Then for the evaluation of the indicator function $I(\tau ;\mbox{\boldmath $\xi$})$ we need to compute the normal derivatives of the special solution $v_{\tau}(\mbox{\boldmath $x$};\mbox{\boldmath $\xi$})$. 
Given the specific rectangular geometry under consideration, the normal vectors $\mbox{\boldmath $\nu$}=(\nu_1,\nu_2)$ reduce to the unit vectors on the edges 
\[
\mbox{\boldmath $\nu$}=
\left\{ \begin{array}{ll}
\vspace{0.3cm}
\pm \mbox{\boldmath $e$}_1 & \text{on } \partial\Omega_{x_1} := \{ \mbox{\boldmath $x$} \in \partial\Omega\ | \ x_1\in \{0,a\},\ 0<x_2<b\}, \\
\vspace{0.3cm}
\pm \mbox{\boldmath $e$}_2 & \text{on } \partial\Omega_{x_2} := \{ \mbox{\boldmath $x$} \in \partial\Omega\ | \ 0<x_1<a,\ x_2\in \{0,b\}\}.\\
\end{array} \right.
\]

The normal derivatives are then given by

\begin{eqnarray*}
\frac{\partial}{\partial\mbox{\boldmath $\nu$}} v_\tau(\xbf;\xibf) &=\mp \tau i \cdot  \frac{ \left((\xbf-\xibf)\cdot(\mbox{\boldmath $e$}_2+i\mbox{\boldmath $e$}_1)\right) ^2}{|\xbf-\xibf|^4}  e^{-\tau \frac{\xbf-\xibf}{|\xbf-\xibf|^2} \cdot (\mbox{\boldmath $e$}_2+i\mbox{\boldmath $e$}_1)}, \text{ on } \partial\Omega_{x_1}, \\
\frac{\partial}{\partial\mbox{\boldmath $\nu$}} v_\tau(\xbf;\xibf) &= \mp \tau \cdot \frac{ \left((\xbf-\xibf)\cdot(\mbox{\boldmath $e$}_1-i\mbox{\boldmath $e$}_2)\right) ^2}{|\xbf-\xibf|^4}  e^{-\tau \frac{\xbf-\xibf}{|\xbf-\xibf|^2} \cdot (\mbox{\boldmath $e$}_2+i\mbox{\boldmath $e$}_1)}, \text{ on } \partial\Omega_{x_2}.
\end{eqnarray*}


Motivated by Theorem \ref{th1} we would need to evaluate equation \eqref{eq:8}
\begin{equation*}
\lim_{\tau\rightarrow\infty}\frac{\log\vert I(\tau ;\mbox{\boldmath $\xi$})\vert}{\tau}=\frac{1}{2s_{\Sigma}(\mbox{\boldmath $\xi$})},
\end{equation*}
but since this is numerically not feasible we rather obtain an approximation  for a set of finite values  $\tau\in[1,T]$, as suggested in \cite{Bruhl2000,Ikehata2000a}. 
This way we can compute the right hand side for a collection of finite $\tau$ and  $\Phi(\xi_1)=\frac{1}{2}S_{\Sigma}^{-1}(\xi_1,\xi_2)$ corresponds to the slope of $\log|I(\tau;\xibf)|$, i.e. computing the approximation
\begin{equation}\label{eqn:estimateSlope}
\frac{\tau}{2s_{\Sigma}(\mbox{\boldmath $\xi$})} \approx  \log|I(\tau;x)|
\end{equation}
is a practical version of the infinite-precision formula \eqref{eq:8}. Estimating the slope of $\tau$ by linear regression gives a robust estimate for $\Phi(\xi_1)$.

\subsection{A probing algorithm for monitoring during RSW}
Let us now formulate a probing algorithm for monitoring the progress during the process of resistance spot welding (RSW). The purpose of the algorithm is to create a profile of the metal slab with indicators of the crack tips and by that evaluate when the procedure has been successful. The proposed algorithm is motivated by Theorem \ref{th1} and the characteristics of the probing function \eqref{eq:probingFunc} outlined in Section \ref{sec:probingAlgoTheo}. 

Let us fix a pressure point $\xbf^*$ where RSW takes place, then we call the process successful if the joined area around $\xbf^*$ is large enough. To determine this we record the first maxima left and right to $\xbf^*$ as $x_L$ and $x_R$, respectively. Then the area of the joined metal pieces is simply $x_R-x_L$, we call this the gap in the following. If the gap is determined as large enough we move the pressure point and perform RSW at the next point $\xbf^*$. The procedure for monitoring at each pressure point is outlined in Algorithm \ref{alg:monitor}.

\begin{algorithm}
	\caption{Probing algorithm for monitoring progress during RSW}
	\label{alg:monitor}
	\begin{algorithmic}
		\Statex
		\Function{monitorRSW}{$\xbf^*$,$\epsilon$,thresh}
		\Let{$\mathcal{T}$}{$[1,T]$}
        \Let{$\mathcal{X}$}{$[0,a]$}
        \Let{$\xi_2$}{$b+\epsilon$}
        \Let{$|\text{gap}|$}{0}
        \While{$|\text{gap}| < $ thresh }
          \State Obtain Cauchy data pair $(g,u)$
          \For{$\xi_1\in \mathcal{X}$}
            \For{$\tau\in\mathcal{T}$}
            	\State Evaluate \eqref{eqn:kelvinIndicator}: $I(\tau;\xibf)$ 
            \EndFor
            \State Estimate $\Phi(\xi_1)$: slope of \eqref{eqn:estimateSlope} by linear regression
          \EndFor
          \Let{($x_L,x_R)$}{maxima left and right to $\xbf^*$} 
          \Let{$|\text{gap}|$}{$x_R-x_L$}
		\EndWhile
        \State \Return{Success at $x^*$, move to next spot}
		\EndFunction
	\end{algorithmic}
\end{algorithm}

\section{Computational experiments}\label{sec:compExperiments}
In this section we will present the resulting profiles that can be computed by the Kelvin transformed enclosure method for an idealized application to RSW. 
The basic setup for all experiments is as follows. We choose the computational domain as $\Omega=[-4,4]\times[-0.2,0.2]$, this corresponds to $a=8$, $b=0.4$, $c=0.2$ in the theory part, note a shift of the center. The input current is given by continuous boundary data as $\varphi(\theta)=\sin(3\theta)$, with an arc-length parametrization for the domain $\Omega$. Furthermore, we let the Neumann data be only partially supported on top and bottom excluding the corner points, such that $\supp(\varphi)\subset\{x\in\bndry:|x_1|<4\}$, note that the arc-length parametrization has to be adjusted to the partial boundary. The FEM mesh is build in the MATLAB PDE toolbox as triangular mesh with 30720 elements, to be fine enough to resolve small cracks. The crack width is chosen as 0.04 in all experiments and is simulated as a cavity with conductivity 0. For evaluating \eqref{eqn:kelvinIndicator} we chose $\tau\in[1,5]$, in intervals of length 0.1, and compute $I(\tau,\xibf)$ for $\xi_1\in[-4,4]$ in steps of length 0.05.

We have observed that the choice of the probing line $\Gamma_\epsilon$, i.e. the tip of the probing discs, is crucial. There are two key observations to take into account: closer to the domain will have a good resolution in the inner part of the domain but produces oscillations closer to the corners of $\Omega$; larger distances lead to a regularizing effect that smooths large oscillations but will reduce the accuracy of the detected tip position. Taking these observations into account we have implemented a varying distance of $\Gamma_\epsilon$ to have good resolution in the middle and little oscillation close to the corners. That is we set $\Gamma_\epsilon(\xi_1)=(\xi_1,\min(2,\max(|\xi_1|/2.5,0.5)))$.

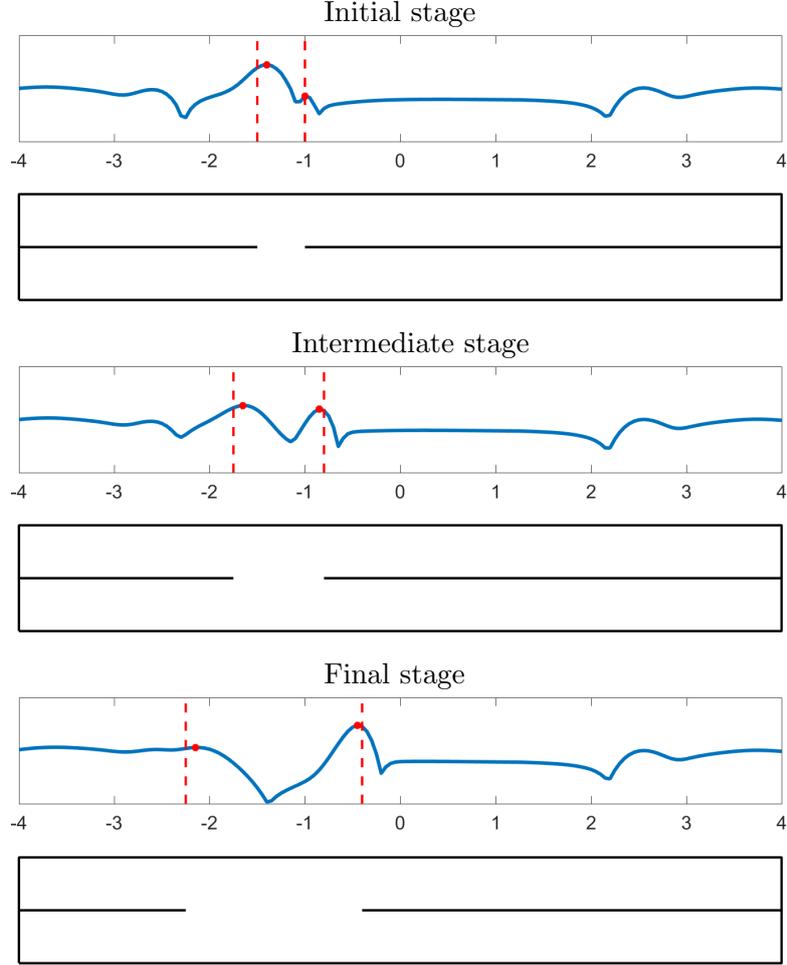
\begin{figure}[ht!]
\centering
\begin{picture}(300,365)
\put(0,300){\includegraphics[width=300pt]{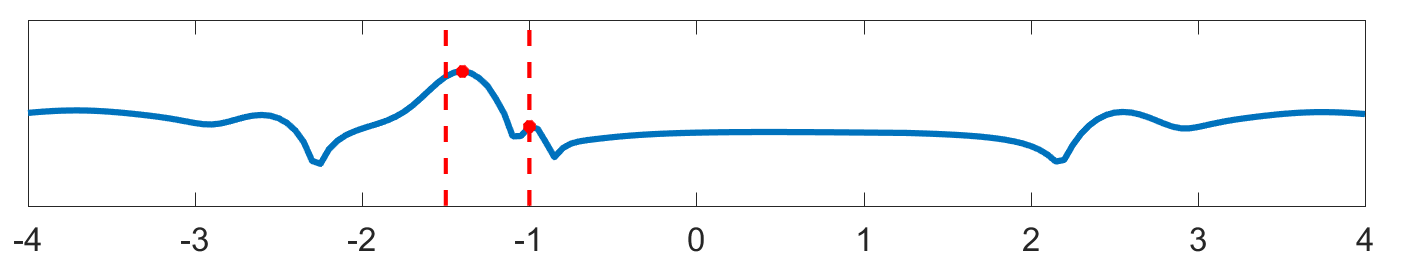}}
\put(0,240){\includegraphics[width=300pt]{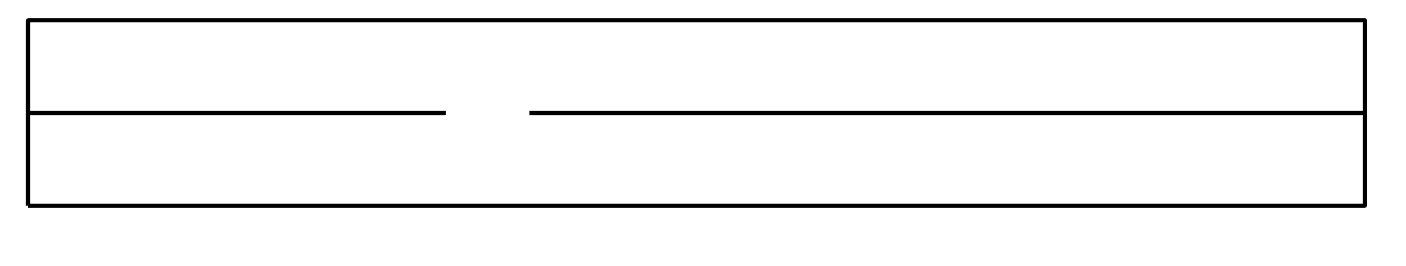}}

\put(0,175){\includegraphics[width=300pt]{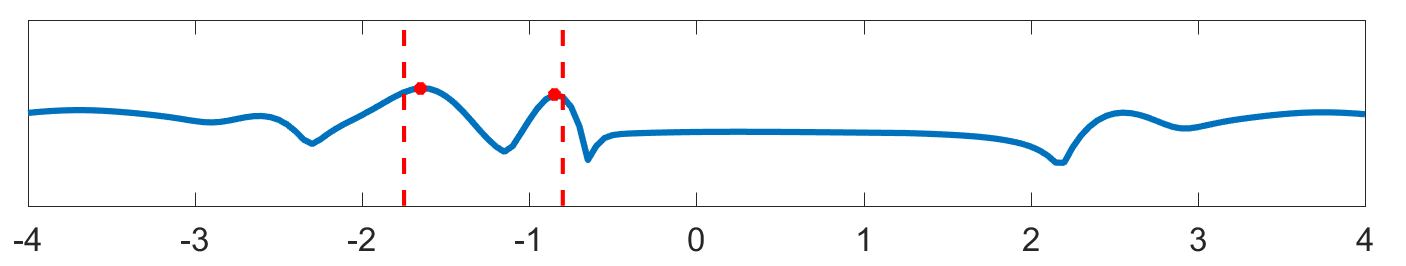}}
\put(0,115){\includegraphics[width=300pt]{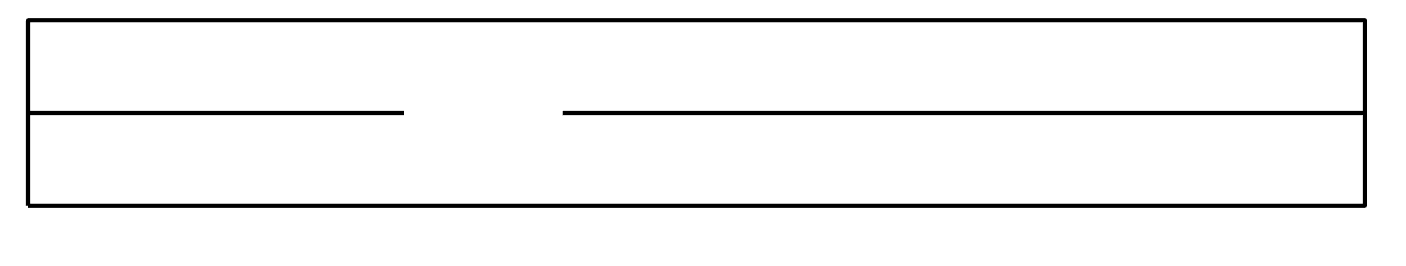}}

\put(0,50){\includegraphics[width=300pt]{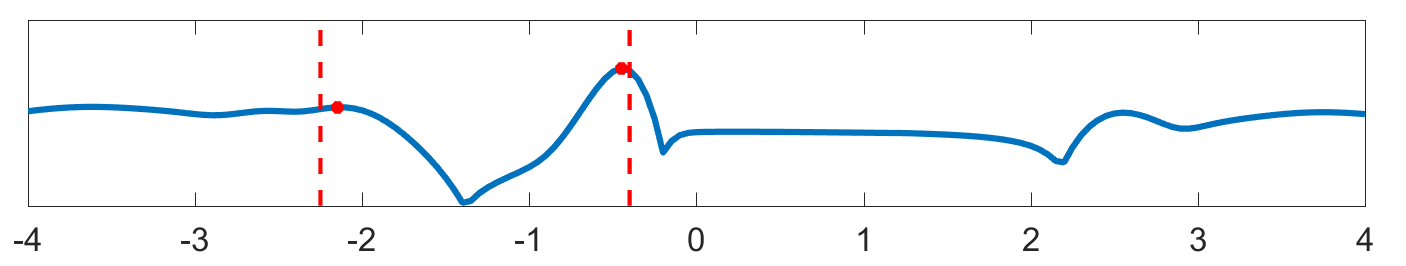}}
\put(0,-10){\includegraphics[width=300pt]{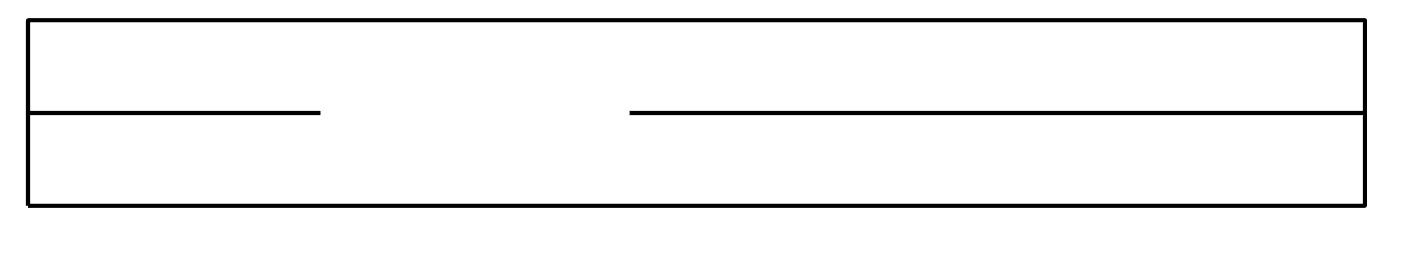}}
\put(120,108){Final stage}
\put(108,233){Intermediate stage}
\put(120,358){Initial stage}

\end{picture}
\caption{\label{fig:firstGap} Computational results for RSW with pressure point $\xbf^*=-1.25$ and a desired gap length $1.5$. Displayed are the profiles obtained from the Kelvin transformed enclosure method paired with the corresponding domain and current stage of the gap. The red lines indicate the true position of the crack tips and the red dots indicate the estimated position.}
\end{figure}

\begin{figure}[ht!]
\centering
\begin{picture}(300,360)
\put(0,300){\includegraphics[width=300pt]{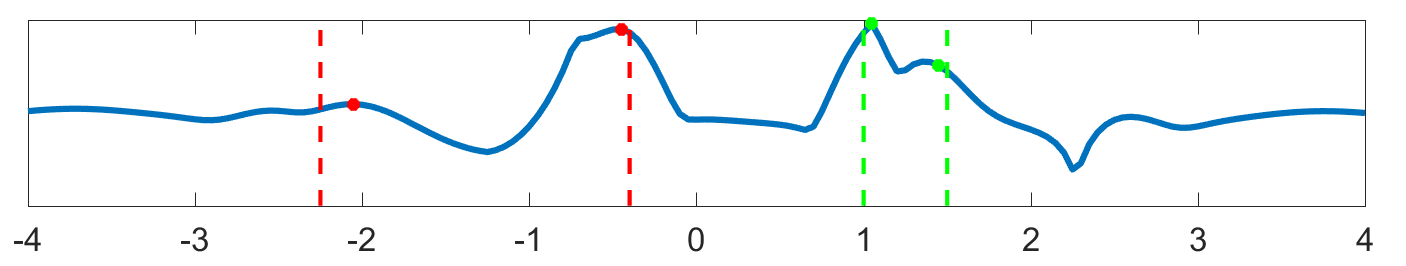}}
\put(0,240){\includegraphics[width=300pt]{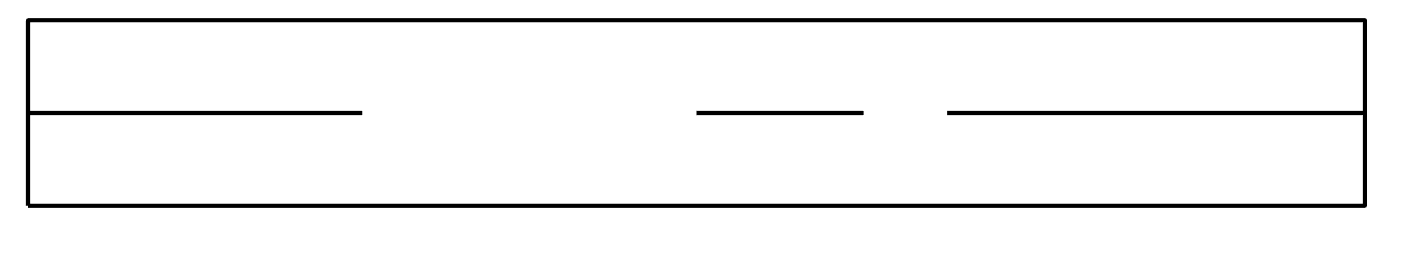}}

\put(0,175){\includegraphics[width=300pt]{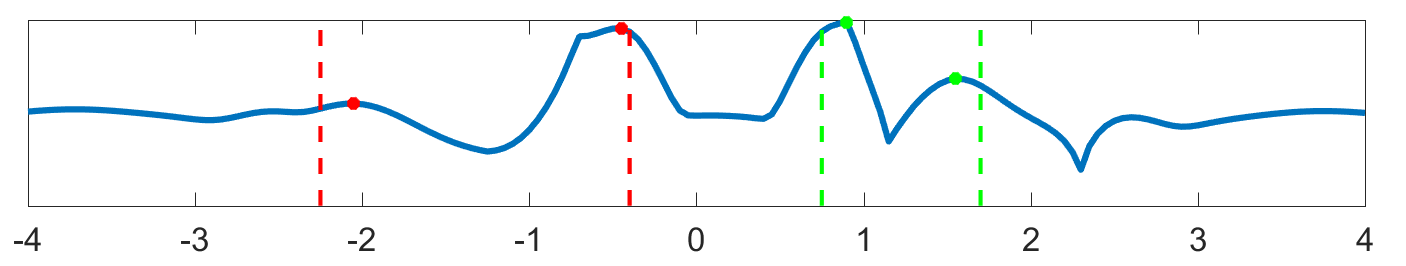}}
\put(0,115){\includegraphics[width=300pt]{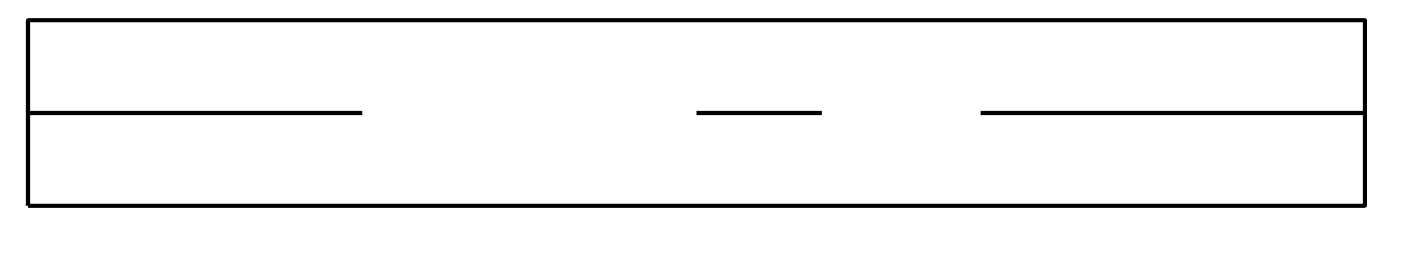}}

\put(0,50){\includegraphics[width=300pt]{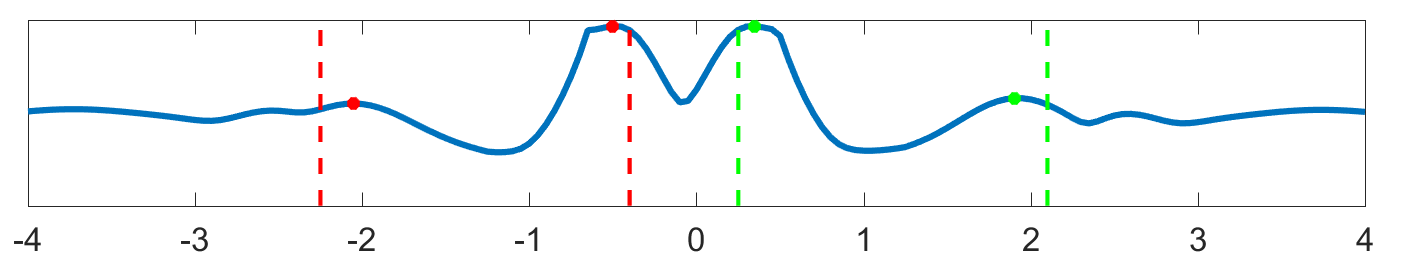}}
\put(0,-10){\includegraphics[width=300pt]{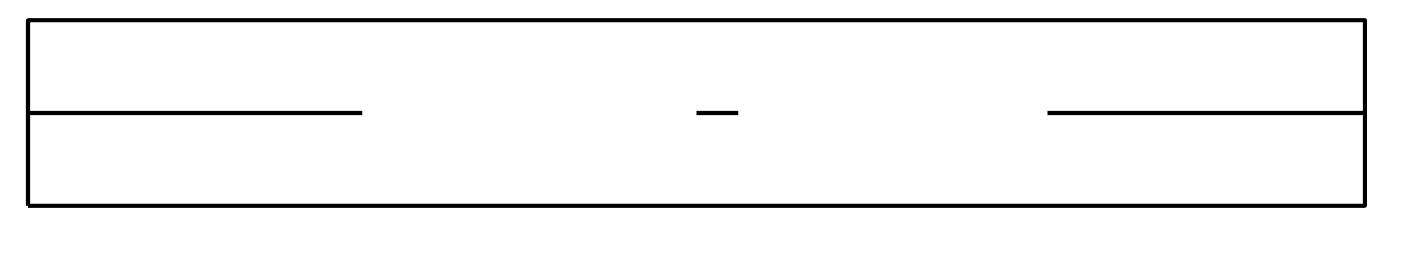}}

\put(120,108){Final stage}
\put(108,233){Intermediate stage}
\put(120,358){Initial stage}

\end{picture}
\caption{\label{fig:secondGap} Computational results for RSW with the second pressure point $\xbf^*=1.25$ and a desired gap length $1.5$. Displayed are the profiles obtained from the Kelvin transformed enclosure method paired with the corresponding domain and current stage of the gap. The green lines indicate the true position of the crack tips and the green dots indicate the estimated position. The red lines and dots indicate the result from the previous pressure point.}
\end{figure}

\subsection{A monitoring example}
We now present a simulated example of a possible monitoring procedure during RSW. We want to join the two metal slabs described above by performing RSW at two pressure points chosen as $\pm1.25$ and a desired gap length of $x_R-x_L=1.5$. In Figure \ref{fig:firstGap}, we illustrate the results obtained by Algorithm \ref{alg:monitor} for three states during the welding process. The first two images show the beginning stage with a small gap, then the second two an intermediate stage while the gap grows larger and the last stage, when the gap is larger than the desired length and the welding process at the first spot can be considered successful. The size of the gap has been increased successively until desired gap length has been exceeded. The evolution of the gap is taken asymmetric and evolves faster to the left in steps of $0.25$ and to the right with size $0.2$.
We then move to a second spot and perform the same procedure there, as illustrated in Figure \ref{fig:secondGap}. The evolution of the gap has been chosen in the same asymmetric way to left and right.

\subsection{Robustness to noise}
In order to investigate robustness to measurement noise, we have added normally distributed random noise to the simulated traces $u|_{\bndry}$. The noise level was chosen as $2\cdot 10^{-4}$ relative to the maximum value of $u|_{\bndry}$, which is a typical noise level in Electrical Impedance Tomography. We have observed that in order to deal with the noise, we need to increase the minimal distance of $\Gamma_\epsilon$ to the domain in order to deal with additional oscillations, specifically we set $\Gamma_\epsilon(\xi_1)=(\xi_1,\min(2,\max(|\xi_1|/2.5,0.75)))$ in the noisy case. This observations emphasizes the regularizing effect of the probing distance. Nevertheless, increasing the probing distance will inevitably lead to a loss of resolution and hence we are not able to resolve the small gaps independently anymore and the appearance of the profile becomes overall smoother with some noise components, as can be seen in Figure \ref{fig:noiseGap}.

\begin{figure}[ht!]
\centering
\begin{picture}(300,180)
\put(0,120){\includegraphics[width=300pt]{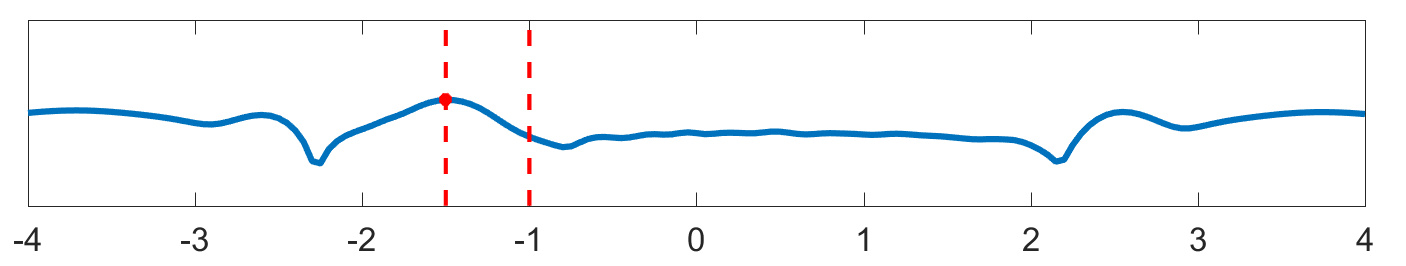}}
\put(0,55){\includegraphics[width=300pt]{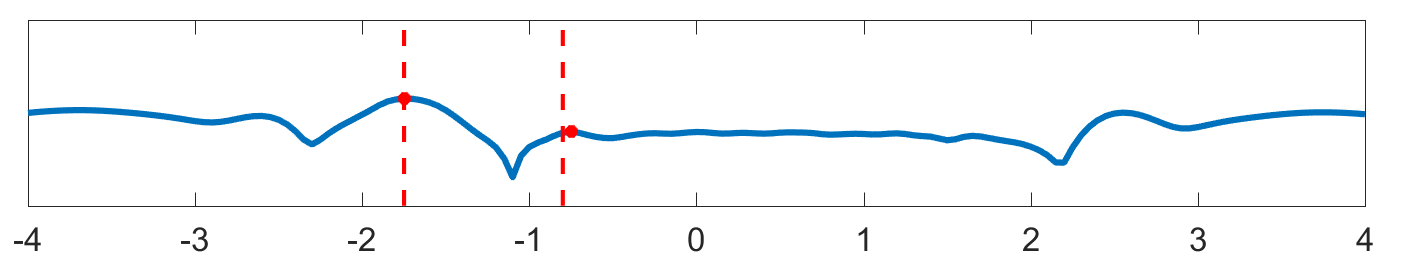}}
\put(0,-10){\includegraphics[width=300pt]{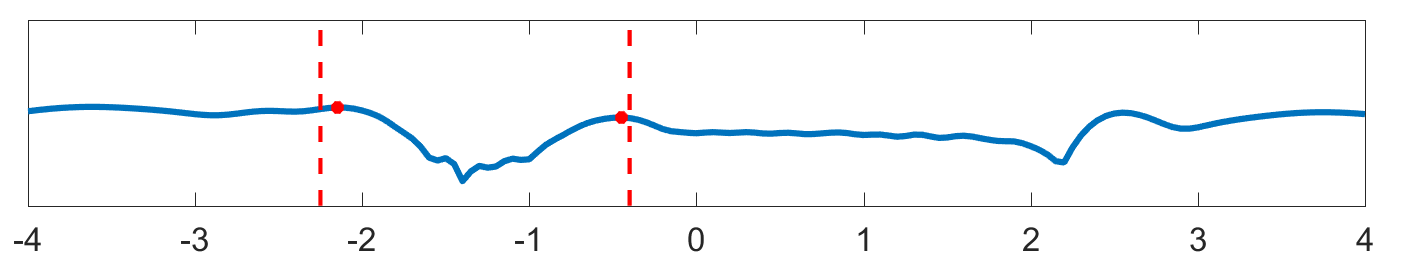}}
\put(120,48){Final stage}
\put(108,113){Intermediate stage}
\put(120,178){Initial stage}

\end{picture}
\caption{\label{fig:noiseGap} Computational results for RSW with pressure point $\xbf^*=-1.25$ and additional relative measurement noise of $2\cdot 10^{-4}$. Displayed are the profiles obtained from the Kelvin transformed enclosure method for the same experiment as in Figure \ref{fig:firstGap}. The red lines indicate the true position of the crack tips and the red dots indicate the estimated position.}
\end{figure}

\subsection{Discussion}
As illustrated in the computational examples we can \linebreak clearly identify the tips of the evolving cracks during the process of resistance spot welding. Even though the position cannot be determined perfectly, the maxima of the profile function give a good indicator of the evolved crack. The estimated locations of the maxima are shown in Table \ref{table:estimateError} and indicate that the procedure is also stable under noise. Due to higher regularization needed (by increasing the distance of the probing function), we are not able to resolve both crack tips for the starting location and get only one maxima in the welding region.

One key observation is that, if the gap is well inside the domain, it is easier to localize than if it is close to the corners of the domain $\Omega$, as is evident in the bottom profile of Figure \ref{fig:secondGap}. Whereas the exact effect causing this is not clear at the current stage, this might be due to interference with the boundary.

\begin{table}[h] 
{\centering
\scriptsize
  \caption{Results of crack tip estimation}
    \begin{tabular}{l l|c|c|c} 
&left gap&{\sc start } &{\sc middle} & {\sc end} \\
    \hline
    \hline 
 & {\sc correct location}    & (-1.50,-1.00) &    (-1.75,-0.80) & (-2.25,-0.40 )   \\
 & {\sc estimated location}  & (-1.40,-1.00) & (-1.65,-0.85) & (-2.15,-0.45)\\   
  & {\sc noisy estimate}  & (-1.50,{\ \ } -- \ ) & (-1.75,-0.75) & (-2.15,-0.45)\\   
     \\
&right gap &  &  &  \\
     \hline
    \hline 
 & {\sc correct location}    & (1.50,1.00) &    (1.70,0.75) & (2.10,0.25)   \\
 & {\sc estimated location}  & (1.35,1.05) &    (1.55,0.90) & (1.90,0.35) \\ 
  & {\sc noisy estimate}  & (1.45,{\ } -- \ ) &    (1.65,0.75) & (1.90,0.20) \\ 

    \end{tabular}%
  \label{table:estimateError} }
\end{table}%

\section{Conclusions}\label{sec:conclusions}
We have shown that the Kelvin transformed enclosure method can be successfully utilized to detect cracks in conductive media with a single measurement. The computational probing procedure relies on creating a profile of the domain, where tips of the cracks are indicated as local maxima. We have also justified this probing procedure of the domain with new theoretical results, that provide a theoretical motivation for our approach.

Initial experiments for an idealized case of resistance spot welding show the effectiveness to locate crack tips inside the domain. Even though the tips can not be perfectly recovered, we are able to monitor the development of the gap during the welding process successfully. 

This study provides the theoretical basis for a monitoring procedure for quality control during welding. In future research we will concentrate on extending the theory to the three dimensional case as well as providing a more realistic computational example that takes geometric restrictions of electrodes located at the boundary into account.

\section*{Acknowledgment}
Authors acknowledge support from Academy of Finland through the \linebreak Finnish Centre of Excellence in Inverse Problems Research
2012--2017 (decision number 284715) and the Finnish Centre of Excellence in Inverse Modelling and
Imaging 2018--2025 (decision numbers 312119 and 312339).
MI was partially supported by Grant-in-Aid for Scientific Research (C)(No. 17K05331) of Japan
Society for the Promotion of Science.
HI was partially supported by Grant-in-Aid for Scientific Research (C)(No. 18K03380) of Japan
Society for the Promotion of Science.

\bibliographystyle{siam}
\bibliography{Inverse_problems_references_2018}

\begin{thebibliography}{10}

\bibitem{Alessandrini1997}
{\sc G.~Alessandrini and E.~DiBenedetto}, {\em Determining 2-dimensional cracks
  in 3-dimensional bodies: Uniqueness and stability}, Indiana University
  Mathematics Journal, 46 (1997), pp.~1--82.

\bibitem{Alessandrini2013}
{\sc G.~Alessandrini and E.~Sincich}, {\em Cracks with impedance; stable
  determination from boundary data}, Indiana University Mathematics Journal, 62
  (2013), pp.~947--989.

\bibitem{Andrieux1996}
{\sc S.~Andrieux and B.~A. Abda}, {\em Identification of planar cracks by
  complete overdetermined data: inversion formulae}, Inverse Problems, 12
  (1996), pp.~553--563.

\bibitem{Aparicio1996}
{\sc N.~D. Aparicio and M.~K. Pidcock}, {\em The boundary inverse problem for
  the {L}aplace equation in two dimensions}, Inverse Problems, 12 (1996),
  pp.~565--577.

\bibitem{Boukari2013}
{\sc Y.~Boukari and H.~Haddar}, {\em The factorization method applied to cracks
  with impedance boundary conditions}, Inverse Problems and Imaging, 7 (2013),
  pp.~1123--1138.

\bibitem{Bruhl2000}
{\sc M.~Br\"uhl and M.~Hanke}, {\em Numerical implementation of two
  non-iterative methods for locating inclusions by impedance tomography},
  Inverse Problems, 16 (2000), pp.~1029--1042.

\bibitem{Bruhl2001b}
{\sc M.~Br\"uhl, M.~Hanke, and M.~Pidcock}, {\em Crack detection using
  electrostatic measurements}, Mathematical Modelling and Numerical Analysis,
  35 (2001), pp.~595--605.

\bibitem{Isakov1995}
{\sc A.~R. Elcrat, V.~Isakov, and O.~Neculoiu}, {\em On finding a surface crack
  from boundary measurements}, Inverse problems, 11 (1995), pp.~343--351.

\bibitem{Eller1996}
{\sc M.~Eller}, {\em Identification of cracks in three-dimensional bodies by
  many boundary measurements}, Inverse Problems, 12 (1996), p.~395.

\bibitem{Friedman1989a}
{\sc A.~Friedman and M.~Vogelius}, {\em Determining cracks by boundary
  measurements}, Indiana University Mathematics Journal, 38 (1989),
  pp.~527--556.

\bibitem{Grisvard1985}
{\sc P.~Grisvard}, {\em Elliptic Problems in Nonsmooth Domains}, Pitman
  Advanced Publishing Program, 1985.

\bibitem{Ikehata1999}
{\sc M.~Ikehata}, {\em Enclosing a polygonal cavity in a two-dimensional
  bounded domain from {C}auchy data}, Inverse Problems, 15 (1999),
  pp.~1231--1241.

\bibitem{Ikehata2000c}
{\sc M.~Ikehata}, {\em Reconstruction of the support function for inclusion
  from boundary measurements}, Journal of Inverse and Ill-Posed Problems, 8
  (2000), pp.~367--378.

\bibitem{Ikehata2003b}
{\sc M.~Ikehata}, {\em Complex geometrical optics solutions and inverse crack
  problems}, Inverse Problems, 19 (2003), pp.~1385--1405.

\bibitem{Ikehata2006}
\leavevmode\vrule height 2pt depth -1.6pt width 23pt, {\em Inverse crack
  problem and probe method}, Cubo, 8 (2006), pp.~29--40.

\bibitem{Ikehata2016}
{\sc M.~Ikehata, H.~Itou, and A.~Sasamoto}, {\em The enclosure method for an
  inverse problem arising from a spot welding}, Mathematical Methods in the
  Applied Sciences, 39 (2016), pp.~3565--3575.

\bibitem{Ikehata2003c}
{\sc M.~Ikehata and G.~Nakamura}, {\em Reconstruction formula for identifying
  cracks}, Journal of Elasticity, 71 (2003), pp.~59--72.

\bibitem{Ikehata2002}
{\sc M.~Ikehata and T.~Ohe}, {\em A numerical method for finding the convex
  hull of inclusions using the enclosure method}, in Electromagnetic
  Nondestructive Evaluation (VI), Studies in Applied Electromagnetics and
  Mechanics, IOS Press, 2002, pp.~21--28.

\bibitem{Ikehata2002c}
{\sc M.~Ikehata and T.~Ohe}, {\em A numerical method for finding the convex
  hull of polygonal cavities using enclosure method}, Inverse Problems, 18
  (2002), pp.~111--124.

\bibitem{Ikehata2008}
\leavevmode\vrule height 2pt depth -1.6pt width 23pt, {\em The enclosure method
  for an inverse crack problem and the {M}ittag-{L}effler function}, Inverse
  Problems, 24 (2008), p.~015006 (27pp).

\bibitem{Ikehata2000a}
{\sc M.~Ikehata and S.~Siltanen}, {\em Numerical method for finding the convex
  hull of an inclusion in conductivity from boundary measurements}, Inverse
  Problems, 16 (2000), pp.~1043--1052.

\bibitem{Ikehata2004}
{\sc M.~Ikehata and S.~Siltanen}, {\em Electrical impedance tomography and
  {M}ittag-{L}effler's function}, Inverse Problems, 20 (2004), pp.~1325--1348.

\bibitem{Olver1997}
{\sc F.~Olver}, {\em Asymptotics and special functions}, AK Peters/CRC Press,
  1997.

\end{thebibliography}
\end{document}